\documentclass[10pt]{amsart}
      \usepackage[mathscr]{eucal}
      \usepackage{amsmath,amsfonts, amssymb}
\def\rg{\hbox to 30pt{\rightarrowfill}}
\def\lg{\hbox to 30pt{\leftarrowfill}}

      \parskip\smallskipamount
          \newtheorem{theorem}{Theorem}[section]
      
      \newtheorem{proposition}[theorem]{Proposition}
      \newtheorem{corollary}[theorem]{Corollary}
      \newtheorem{lemma}[theorem]{Lemma}

      \makeatletter
      \@addtoreset{equation}{section}
      \makeatother

      \newcommand{\BB}{{\mathbb B}}
      \newcommand{\CC}{{\mathbb C}}
      \newcommand{\NN}{{\mathbb N}}

      \newcommand{\FF}{{\mathbb F}}
      \newcommand{\TT}{{\mathbb T}}

      \newcommand{\cA}{{\mathcal A}}
      
      \newcommand{\cC}{{\mathcal C}}
      \newcommand{\cD}{{\mathcal D}}
      \newcommand{\cE}{{\mathcal E}}
      
      \newcommand{\cG}{{\mathcal G}}
      \newcommand{\cH}{{\mathcal H}}
      
      \newcommand{\cK}{{\mathcal K}}
      \newcommand{\cL}{{\mathcal L}}
      \newcommand{\cM}{{\mathcal M}}
      \newcommand{\cN}{{\mathcal N}}
      \newcommand{\cP}{{\mathcal P}}
      \newcommand{\cR}{{\mathcal R}}

      \newcommand{\cU}{{\mathcal U}}

      \newcommand{\rank}{\hbox{\rm{rank}}\,}

      \newdimen\expt
      \def\boxit#1{\setbox0\hbox{$\displaystyle{#1}$}
            \hbox{\lower.4\expt
       \hbox{\lower3\expt\hbox{\lower\dp0
            \hbox{\vbox{\hrule height.4\expt
       \hbox{\vrule width.4\expt\hskip3\expt
            \vbox{\vskip3\expt\box0\vskip2\expt}%
       \hskip3\expt\vrule width.4\expt}\hrule height.4\expt}}}}}}
      
\begin{document}
       \pagestyle{myheadings}
      \markboth{ Gelu Popescu}{ Unitary invariants  on the unit ball of $B(\cH)^n$ }

      \title [ Unitary invariants for  on the unit ball of $B(\cH)^n$ ]
{ Unitary invariants on the unit ball of $B(\cH)^n$ }
        \author{Gelu Popescu}
\date{February 5, 2012}
      \thanks{Research supported in part by an NSF grant}

      \address{Department of Mathematics, The University of Texas
      at San Antonio \\ San Antonio, TX 78249, USA}
      \email{\tt gelu.popescu@utsa.edu}

\bigskip

\begin{abstract}
 In this paper, we introduce a
unitary invariant
$$\Gamma:[B(\cH)^n]_1^-\to \NN_\infty\times \NN_\infty\times \NN_\infty,\qquad
  \NN_\infty:=\NN\cup \{ \infty\},
$$
 defined in terms of the characteristic
function $\Theta_T$, the noncommutative Poisson kernel $K_T$, and
the defect operator $\Delta_T$ associated with $T\in
[B(\cH)^n]_1^-$. We show that the map $\Gamma$  detects the pure row
isometries in the closed  unit ball of $B(\cH)^n$ and completely
classify them up to a unitary equivalence. We  also show that
$\Gamma$
 detects
  the pure row contractions with polynomial characteristic
functions and  completely non-coisometric row contractions, while
the pair $(\Gamma, \Theta_T)$
 is a complete unitary invariant for these   classes of
 row contractions.

The unitary invariant $\Gamma$ is extracted from  the theory of
characteristic functions and noncommutative Poisson transforms, and
from the geometric structure of row contractions with polynomial
characteristic functions which are studied in this paper.  As an
application, we characterize the  row contractions  with constant
characteristic function. In particular, we show that any completely non-coisometric
row contraction $T$ with constant characteristic function  is
homogeneous, i.e., $T$ is unitarily equivalent to $\varphi(T)$ for
any free holomorphic automorphism  $\varphi$  of the unit ball of
$B(\cH)^n$.

   Under a natural topology, we prove that the free holomorphic automorphism group
    $\text{\rm Aut}(B(\cH)^n_1)$ is  a metrizable, $\sigma$-compact,  locally compact group,
     and provide  a  concrete unitary projective representation of it in terms
      of noncommutative Poisson kernels.
\end{abstract}

      \maketitle

\bigskip

\bigskip
\bigskip





\bigskip

\section*{Introduction}

An $n$-tuple $T=(T_1,\ldots,T_n)$ of bounded linear operators
is called row contraction if  it belongs to the closed unit ball
 $$
 [B(\cH)^n]_1^-:=\{(X_1,\ldots, X_n)\in B(\cH)^n : \ X_1X_1^*+\cdots+X_nX_n^*\leq I\},
 $$
where $B(\cH)$ is the algebra of bounded linear operators on  a
 Hilbert space $\cH$.
In recent  years, there has been exciting progress in multivariable operator
theory on $[B(\cH)^n]_1^-$,
 especially in connection with dilation  theory and unitary invariants for
 $n$-tuples of operators such as characteristic function, curvature and Euler
  characteristic, entropy,  joint numerical radius and  joint $\rho$-operator
  radius
  (see \cite{Po-charact},
\cite{Po-curvature}, \cite{Po-entropy},  \cite{Po-unitary},
\cite{Po-automorphism} and the references therein).

A central problem in multivariable operator theory is the
classification, up to a unitary equivalence, of $n$-tuples of
operators.
 In this paper, we introduce a new
unitary invariant
$$\Gamma:[B(\cH)^n]_1^-\to \NN_\infty\times \NN_\infty\times \NN_\infty
$$
which is extracted from
from the geometric structure of row contractions with polynomial
characteristic functions and the theory of
  noncommutative Poisson transforms on the unit ball of $B(\cH)^n$.
  We use $\Gamma$ to detect  and classify  certain classes of $n$-tuples of  operators
  in the unit ball of $B(\cH)^n$.

In Section 1,  we  show that a row contraction $T=(T_1,\ldots,T_n)$
has polynomial characteristic function  of degree $m\in
\NN:=\{0,1,\ldots\}$ if and only if  $T_i$ admits a canonical upper
triangular representation
$$
T_i=\left[\begin{matrix}V_i&*&*\\
0&N_i&*\\
0&0&W_i
\end{matrix}\right],\qquad i=1,\ldots,n,
$$
where $(V_1,\ldots, V_n)$ is a pure isometry, $(N_1,\ldots, N_n)$ is
a nilpotent row contraction of order $m$, and $(W_1,\ldots,W_n)$ is
a coisometry. In the particular case when $n=1$ and $T$ is a completely non-unitary
(c.n.u.) contraction, we recover a recent result of Foia\c s and Sarkar \cite{FS}.
The results of Section 1 lead to the definition of the map
$$\Gamma:[B(\cH)^n]_1^-\to \NN_\infty\times \NN_\infty\times \NN_\infty,\qquad  \Gamma(T):=(p,m,q)
$$
 by  setting  $m:=\deg (\Theta_T)$, \  $q:=\dim (\ker K_T)$, and
 $$p:=\begin{cases} \dim (\cD_m\ominus \cD_{m+1})& \quad \text{ if  } m\in \NN\\
 \dim \overline{\Delta_T\cH}& \quad \text{ if }
m=\infty,
\end{cases}$$
  where
 $
 \cD_m:=\overline{\text{\rm span}}\{T_\beta \Delta_Th:\ h\in \cH, |\beta|\geq
m\}, $
 $\Theta_T$ is the characteristic function, $K_T$ is the
noncommutative Poisson kernel, and $\Delta_T$ is
the defect operator  associated with $T\in
[B(\cH)^n]_1^-$.

In Section 2, we show that the map $\Gamma$  detects the pure row
isometries in the closed  unit ball of $B(\cH)^n$ and completely
classify them up to a unitary equivalence. We  also show that
$\Gamma$
 detects
  the pure row contractions with polynomial characteristic
functions and  completely non-coisometric  (c.n.c.) row
contractions, while the pair $(\Gamma, \Theta_T)$
 is a complete unitary invariant for these   classes of
 row contractions.
As an application of the results from Section 1, we prove that
 the characteristic function $\Theta_T$ is a constant if and only if
 $T$
   admits  a canonical upper triangular representation
$$
T_i=\left[\begin{matrix}V_i&*\\
0&W_i
\end{matrix}\right],\qquad i=1,\ldots,n,
$$
where $V:=(V_1,\ldots, V_n)$ is a pure  isometry and
$W:=(W_1,\ldots,W_n)$ is a coisometry.

In Section 3, we prove that  a   c.n.c  row contraction  $T$ is
homogeneous if and only if $\Theta_T\circ \Psi^{-1}$
 coincides with the characteristic function $\Theta_T$ for any
 $\Psi$ in the group
$\text{\rm Aut}(B(\cH)^n_1)$ of   free holomorphic automorphisms of
$[B(\cH)^n]_1$. In particular, we show that any c.n.c row
contraction $T$ with constant characteristic function  is
homogeneous, i.e., $T$ is unitarily equivalent to $\varphi(T)$ for
any  $\varphi\in \text{\rm Aut}(B(\cH)^n_1)$.   Moreover, we show
that
$$
\varphi_i(T)=U_\varphi T_i U_\varphi^*,\qquad i=1,\ldots,n,
$$
  where $U_\varphi$ is a
unitary operator  satisfying  relation $ U_\varphi
U_\psi=c(\varphi, \psi) U_{\varphi\circ \psi}$ for some
  complex number $c(\varphi, \psi)\in \TT$. We remark that in the single
  variable case $(n=1)$ we find again some of the results obtain by Clark, Misra, and Bagchi
  (see \cite{CM}, \cite{BM1}).

The theory of  characteristic functions for row contractions
\cite{Po-charact}
  was used  in \cite{Po-automorphism} to determine the group
$\text{\rm Aut}(B(\cH)^n_1)$ of all free holomorphic automorphisms
of $[B(\cH)^n]_1$. We obtained a characterization
 of the unitarily implemented automorphisms  of the
 Cuntz-Toeplitz algebra $C^*(S_1,\ldots, S_n)$,  which leave invariant
 the noncommutative disc algebra $\cA_n$, in terms of noncommutative Poisson transforms.
  This result provided new insight
  into   Voiculescu's group \cite{Vo} of automorphisms   of the Cuntz-Toeplitz
 algebra and  revealed new  connections  with    noncommutative  multivariable
 operator theory.
Employing some techniques from \cite{Po-automorphism}, we prove
that, with respect to the metric
$$
d_\cE(\phi,\psi):=\|\phi -\psi\|_\infty +
\|\phi^{-1}(0)-\psi^{-1}(0)\|, \qquad \phi,\psi\in  \text{\rm Aut}(B(\cH)^n_1),
$$
  the free holomorphic automorphism group
    $\text{\rm Aut}(B(\cH)^n_1)$ is  a  $\sigma$-compact,  locally compact group,
     and  we provide  a  concrete unitary projective representation of it in terms
      of noncommutative Poisson kernels.

      The author thanks Jaydeb Sarkar for useful discussions on the
      subject of this paper.

\bigskip

\section{   Row contractions  with  polynomial  characteristic functions  }

Let $H_n$ be an $n$-dimensional complex  Hilbert space with
orthonormal
      basis
      $e_1$, $e_2$, $\dots,e_n$, where $n=1,2,\dots$.
       We consider the full Fock space  of $H_n$ defined by
      $$F^2(H_n):=\CC1\oplus \bigoplus_{k\geq 1} H_n^{\otimes k},$$
      where  $H_n^{\otimes k}$ is the (Hilbert)
      tensor product of $k$ copies of $H_n$.
      Define the left  (resp.~right) creation
      operators  $S_i$ (resp.~$R_i$), $i=1,\ldots,n$, acting on $F^2(H_n)$  by
      setting
      $$
       S_i\varphi:=e_i\otimes\varphi, \quad  \varphi\in F^2(H_n),
      $$
       (resp.~$
       R_i\varphi:=\varphi\otimes e_i, \quad  \varphi\in F^2(H_n)$).
        The noncommutative disc algebra $\cA_n$ (resp.~$\cR_n$) is the
norm closed algebra generated by the left (resp.~right) creation
operators and the identity. The   noncommutative analytic Toeplitz
algebra $F_n^\infty$ (resp.~$\cR_n^\infty$)
 is the  weakly
closed version of $\cA_n$ (resp.~$\cR_n$). These algebras were
introduced  (see  \cite{Po-von}, \cite{Po-funct}, \cite{Po-disc}) in
connection with a noncommutative von Neumann  type inequality
\cite{von}.

 Let $\FF_n^+$ be the unital free semigroup on $n$ generators
$g_1,\ldots, g_n$ and the identity $g_0$.  The length of $\alpha\in
\FF_n^+$ is defined by $|\alpha|:=0$ if $\alpha=g_0$ and
$|\alpha|:=k$ if
 $\alpha=g_{i_1}\cdots g_{i_k}$, where $i_1,\ldots, i_k\in \{1,\ldots, n\}$.
If $(X_1,\ldots, X_n)\in B(\cH)^n$, where $B(\cH)$ is the algebra of
all bounded linear operators on the Hilbert space $\cH$,    we set
$X_\alpha:= X_{i_1}\cdots X_{i_k}$  and $X_{g_0}:=I_\cH$. We denote
$e_\alpha:= e_{i_1}\otimes\cdots \otimes  e_{i_k}$ and $e_{g_0}:=1$.

  We recall   (\cite{Po-charact},
      \cite{Po-analytic})
       a few facts
       concerning multi-analytic   operators on Fock
      spaces.
         We say that
       a bounded linear
        operator
      $M$ acting from $F^2(H_n)\otimes \cK$ to $ F^2(H_n)\otimes \cK'$ is
       multi-analytic with respect to $S_1,\ldots, S_n$
      if
      \begin{equation*}
      M(S_i\otimes I_\cK)= (S_i\otimes I_{\cK'}) M, \qquad
        i=1,\dots, n.
      \end{equation*}
       We can associate with $M$ a unique formal Fourier expansion
      \begin{equation*}       M(R_1,\ldots, R_n):= \sum_{\alpha \in \FF_n^+}
      R_\alpha \otimes \theta_{(\alpha)}, \end{equation*}
where $\theta_{(\alpha)}\in B(\cK, \cK')$.
       We  know  that
        $$M =\text{\rm SOT-}\lim_{r\to 1}\sum_{k=0}^\infty
      \sum_{|\alpha|=k}
         r^{|\alpha|} R_\alpha\otimes \theta_{(\alpha)},
         $$
         where, for each $r\in [0,1)$, the series converges in the operator norm.
      Moreover, the set of  all multi-analytic operators in
      $B(F^2(H_n)\otimes \cK,
      F^2(H_n)\otimes \cK')$  coincides  with
      $R_n^\infty\bar \otimes B(\cK,\cK')$,
      the WOT-closed operator space generated by the spatial tensor
      product.
A multi-analytic operator is called inner if it is an isometry. We
remark that  similar results are valid  for  multi-analytic
operators with respect to the right creation operators $R_1,\ldots,
R_n$.

 According to \cite{Po-holomorphic}, a map $F:[B(\cH)^n]_{1}\to
  B( \cH)\bar\otimes_{min} B(\cE, \cG)$ is called
  {\it free
holomorphic function} on  $[B(\cH)^n]_{\gamma}$, $\gamma>0$,  with coefficients
in $B(\cE, \cG)$ if there exist $A_{(\alpha)}\in B(\cE, \cG)$,
$\alpha\in \FF_n^+$, such that
$$
F(X_1,\ldots, X_n)=\sum\limits_{k=0}^\infty \sum\limits_{|\alpha|=k}
X_\alpha\otimes  A_{(\alpha)},
$$
where the series converges in the operator  norm topology  for any
$(X_1,\ldots, X_n)\in [B(\cH)^n]_{\gamma}$, where
$$
 [B(\cH)^n]_\gamma:=\{(X_1,\ldots, X_n)\in B(\cH)^n : \ \|X_1X_1^*+\cdots+X_nX_n^*\|^{1/2} < \gamma \},
 $$

For simplicity, throughout this paper, $[X_1,\ldots,X_n]$ denotes
either the $n$-tuple $(X_1,\ldots,X_n)\in B(\cH)^n$ or the  operator
row matrix $[X_1\, \cdots\, X_n]$ acting from $\cH^{(n)}$, the
direct sum of $n$ copies of a Hilbert space $\cH$, to $\cH$. The
{\it characteristic function} associated with an arbitrary row
contraction $T:=[T_1,\ldots, T_n]$, \ $T_i\in B(\cH)$, was
introduced in \cite{Po-charact} (see \cite{SzFBK-book} for the
classical case $n=1$) and it was proved to be  a complete unitary
invariant for completely non-coisometric   row contractions. The
characteristic function  of $T$ is a multi-analytic operator with
respect to $S_1,\ldots, S_n$,
$$
\tilde{\Theta}_T:F^2(H_n)\otimes \cD_{T^*}\to F^2(H_n)\otimes \cD_T,
$$
with the formal Fourier representation
\begin{equation*}
\begin{split}
 \Theta_T(R_1,\ldots, R_n):= -I_{F^2(H_n)}\otimes T|_{\cD_{T^*}}+
\left(I_{F^2(H_n)}\otimes \Delta_T\right)&\left(I_{F^2(H_n)\otimes
\cK}
-\sum_{i=1}^n R_i\otimes T_i^*\right)^{-1}\\
&\left[R_1\otimes I_\cK,\ldots, R_n\otimes I_\cK \right]
\left(I_{F^2(H_n)}\otimes \Delta_{T^*}|_{\cD_{T^*}}\right),
\end{split}
\end{equation*}
where $R_1,\ldots, R_n$ are the right creation operators on the full
Fock space $F^2(H_n)$.
 Here,  we need to clarify some notations since some of them are different
 from those considered in \cite{Po-charact}.
The defect operators  associated with a row contraction
$T:=[T_1,\ldots, T_n]$ are
\begin{equation*}
 \Delta_T:=\left( I_\cH-\sum_{i=1}^n
T_iT_i^*\right)^{1/2}\in B(\cH) \quad \text{ and }\quad
\Delta_{T^*}:=(I-T^*T)^{1/2}\in B(\cH^{(n)}),
\end{equation*}
while the defect spaces are $\cD_T:=\overline{\Delta_T\cH}$ and
$\cD_{T^*}:=\overline{\Delta_{T^*}\cH^{(n)}}$.
Using the $F_n^\infty$-functional calculus  for row contractions \cite{Po-funct},
one can define
$$
\Theta_T(X_1,\ldots, X_n):=\text{\rm SOT-}\lim_{r\to 1}\Theta_T
(rX_1,\ldots, rX_n)
$$
for any c.n.c.  row contraction $(X_1,\ldots, X_n)\in
[B(\cG)^n]_1^-$, where $\cG$ is a Hilbert space. Depending on $T$,
the map $\Theta_T$ may be well-defined on a larger subset of
$B(\cG)^n$. For example, if $\|T\|<1$, then $X\mapsto \Theta_T (X)$
is  a free holomorphic function on the open ball
$[B(\cG)^n]_{\gamma}$, where $\gamma:=\frac{1}{\|T\|}$. Therefore,
the characteristic function  $\tilde \Theta_T$ generates a bounded
free holomorphic function $\Theta_T$ (also called characteristic
function) on $[B(\cG)^n]_1$ with operator-valued coefficients in
$B(\cD_{T^*}, \cD_T)$. Note also that
\begin{equation*}
\begin{split}
 \Theta_T(X_1,\ldots, X_n)= -I_{\cG}\otimes (T|_{\cD_{T^*}})+
\left(I_\cG \otimes \Delta_T\right)&\left(I_{\cG\otimes
\cK}-\sum_{i=1}^n X_i\otimes
T_i^*\right)^{-1}\\
&\left[X_1\otimes I_\cK,\ldots, X_n\otimes I_\cK \right]
\left(I_\cG\otimes \Delta_{T^*}|_{\cD_{T^*}}\right)
\end{split}
\end{equation*}
for any $(X_1,\ldots, X_n)\in [B(\cG)^n]_1$.  The characteristic
function $\tilde{\Theta}_T$ is the model boundary function of $\Theta_T$
with respect to $R_1,\ldots, R_n$ in the sense that
$$\tilde{\Theta}_T=\text{\rm SOT-}\lim_{r\to 1}\Theta_T(rR_1,\ldots,
rR_n),$$
 where  $\Theta(rR_1,\ldots, rR_n)$ is
in $\cR_n\otimes_{min} B(\cK)$ for any $r\in [0,1)$.

Let $T:=[T_1,\ldots, T_n]$  be a row contraction with $T_i\in
B(\cH)$ and consider the subspace $\cH_c\subseteq \cH$ defined by
\begin{equation} \label{hc}
 \cH_c:=\left\{h\in \cH:\
\sum_{|\alpha|=k} \|T_\alpha^*h\|^2=\|h\|^2 \text{ for any }
k=1,2,\ldots\right\}.
\end{equation}
We call $T$  a {\it completely non-coisometric} (c.n.c.)  row
contraction if $\cH_c=\{0\}$. We proved in \cite{Po-isometric} that
$\cH_c$ is a joint invariant subspace under the operators
$T_1^*,\ldots, T_n^*$, and it is  also the largest subspace in $\cH$
on which $T^*$  acts isometrically. Consequently, we have the
following triangulation with respect to the decomposition
$\cH=\cH_c\oplus \cH_{cnc}$:
$$
T_i=\left(\begin{matrix}A_i&0\\
*&B_i
\end{matrix}\right),\qquad i=1,\ldots,n,
$$
where $*$ stands for an unspecified  entry, $[A_1,\ldots, A_n]$ is a
coisometry, i.e., $A_1 A_1^*+\cdots +A_n A_n^*=I_{\cH_c}$, and
$[B_1,\ldots, B_n]$ is a c.n.c. row contraction. We say that a row
contraction $T$ is {\it pure} if
$$\lim_{k\to\infty}\sum_{\gamma\in \FF_n^+, |\gamma|=k}\|T_{\gamma}^*h\|^2=0,\qquad
h\in \cH.
$$
An $n$-tuple $N:=(N_1,\ldots, N_n)\in B(\cH)^n$ is called {\it
nilpotent} if there is    $m\in \NN$ such that   $N_\alpha=0$ for
all $\alpha\in \FF_n^+$ with $|\alpha|=m$. The order  of a nilpotent
$n$-tuple $N$ is the smallest $m\in \NN$ with the above-mentioned
property.
 Throughout this paper, we make the convention that de degree of a constant polynomial (including the zero polynomial) is zero.
\begin{theorem}
\label{structure} Let $T:=[T_1,\ldots, T_n]\in [B(\cH)^n]_1^-$ be a
 row contraction such that the characteristic  function
$\Theta_T$ is a noncommutative polynomial of degree $m\in \NN$. Then
there exist subspaces  $\cH_v$, $\cH_{nil}$, and $\cH_c$ of $\cH$  such that
$\cH=\cH_v\oplus \cH_{nil}\oplus \cH_c$ and each $T_i$ admits a representation
$$
T_i=\left[\begin{matrix}V_i&*&*\\
0&N_i&*\\
0&0&W_i
\end{matrix}\right],\qquad i=1,\ldots,n,
$$
where $[V_1,\ldots, V_n]\in [B(\cH_v)^n]_1^-$ is a pure isometry,
 $[N_1,\ldots, N_n]\in [B(\cH_{nil})^n]_1^-$ is a nilpotent row
contraction of order $\leq m$, and $[W_1,\ldots,W_n]\in [B(\cH_c)^n]_1^-$
is a coisometry. Moreover, if $m=0$, then $\cH_{nil}=\{0\}$ and
$T_i$ admits the representation
$$
T_i=\left[\begin{matrix}V_i&*\\
0&W_i
\end{matrix}\right],\qquad i=1,\ldots,n,
$$
with respect to the decomposition $\cH=\cH_v\oplus \cH_c$.
\end{theorem}
\begin{proof}

The characteristic function
  $\Theta_T: [B(\cG)^n]_1
\to  B( \cG)\bar\otimes_{min} B(\cD_{T^*}, \cD_T)$ is a bounded free holomorphic function
 given by
$$
\Theta_T(X_1,\ldots, X_n)= -I_{\cG}\otimes
(T|_{\cD_{T^*}})+\sum_{i=1}^n \sum_{k=0}^\infty \sum_{|\alpha|=k}
X_\alpha X_i  \otimes \Delta_{T}(T_{\widetilde\alpha})^*P_i
\Delta_{T^*}|_{\cD_{T^*}}
$$
for  $X=(X_1,\ldots, X_n)\in [B(\cG)^n]_1$,
where  the convergence is in the operator norm  and $P_i$ denotes
the orthogonal projection of $ \cH^{(n)}$ onto the $i$-component of
$\cH^{(n)}$.   Assume that $\Theta_T$ is a noncommutative polynomial
of degree $m\in \NN=\{0,1,\ldots\}$. Then we have
$\Delta_T(T_\beta)^*P_i\Delta_{T^*}=0$ for all $\beta\in \FF_n^+$
with $|\beta|\geq m$ and $i=1,\ldots,n$. Hence, we deduce that
\begin{equation}
\label{Ji} \Delta_{T^*}^2J_iT_\beta \Delta_T=0,\qquad |\beta|\geq m,
\ i=1,\ldots,n,
\end{equation}
where $J_i:\cH\to \cH^{(n)}$ is the injection $J_i h:= \oplus_{j=1}^n\delta_{ji}h$.
Define the subspace
$$
\cH_v:=\overline{\text{\rm span}}\{T_\beta h:\ h\in \cD_T, |\beta|\geq m\}
$$
and note that it is invariant under each operator $T_1,\ldots, T_n$.
In what follows,  we show that  the $n$-tuple
$[T_1|_{\cH_v},\ldots,T_n|_{\cH_v}]\in [B(\cH_v)^n]_1^-$ is an
isometry. Note that if $h\in \cH$, $|\beta|\geq m$, and
$i=1,\ldots,n$, then relation \eqref{Ji} implies
\begin{equation*}
\Delta_{T^*}^2J_iT_\beta \Delta_Th=
\left[\begin{matrix}I-T_1^*T_1 &-T_1^*T_2&\cdots &-T_1^*T_n\\
-T_2^* T_1&I-T_2^*T_2&\cdots &-T_2^*T_n\\
\vdots&\vdots&\vdots&\vdots\\
-T_i^*T_1&-T_i^*T_2& I-T_i^*T_i& -T_i^*T_n\\
\vdots&\vdots&\vdots&\vdots\\
-T_n^*T_1&-T_n^*T_2&\cdots &I-T_n^*T_n
\end{matrix}\right]
\left[\begin{matrix}
0\\0\\\vdots\\ T_\beta \Delta_Th\\\vdots\\0
\end{matrix}\right]
=
\left[\begin{matrix}
-T_1^*T_iT_\beta \Delta h\\
-T_2^*T_iT_\beta \Delta h\\\vdots\\ (I-T_i^*T_i)T_\beta \Delta_Th\\\vdots\\-T_n^*T_iT_\beta \Delta_Th
\end{matrix}\right]=0.
\end{equation*}
Consequently, we have
$$T_j^*T_iT_\beta \Delta_Th=0,\qquad   i,j\in \{1,\ldots,n\}, i\neq j \ \text { and } |\beta|\geq m,
$$
 and
$$(I-T_i^*T_i)T_\beta \Delta_Th=0,\qquad i\in \{1,\ldots,n\}.
$$
Hence, we deduce that $T_i(\cH_v)\perp T_j(\cH_v)$ if $i\neq j$ and
$\|T_ix\|=\|x\|$ for any $x\in \cH_v$. Therefore,  the $n$-tuple
$[T_1|_{\cH_v},\ldots,T_n|_{\cH_v}]\in [B(\cH_v)^n]_1^-$ is an
isometry. Set $V_i:=T_i|_{\cH_v}:\cH_v\to \cH_v$ for $i=1,\ldots,n$.
According to the Wold decomposition for isometries with orthogonal
ranges  (see \cite{Po-isometric}), there is a unique orthogonal
decomposition $\cH_v=\cH_s\oplus \cH_u$ such that $\cH_u$ and
$\cH_s$ are reducing subspaces under $V_1,\ldots V_n$, the $n$-tuple
$[V_1|_{\cH_s},\ldots, V_n|_{\cH_s}]$ is a pure row isometry and
$[V_1|_{\cH_u},\ldots, V_n|_{\cH_u}]$ is a Cuntz isometry, i.e.,
$\sum_{i=1}^n (V_i|_{\cH_u})(V_i|_{\cH_u})^*=I_{\cH_u}$. Moreover,
we have
$$\cH_u=\{h\in \cH_v:\ \sum_{|\alpha|=k} \|V_\alpha^*h\|^2=\|h\|^2 \
\text{ for all } \ k\in \NN\}.
$$
Note that,  since $T=[T_1,\ldots, T_n]$ is a row contraction, if
$h\in \cH_u$, then
$$
\|h\|^2=\sum_{|\alpha|=k} \|V_\alpha^*h\|^2=\sum_{|\alpha|=k}
\|P_{\cH_v} T_\alpha^*h\|^2\leq \sum_{|\alpha|=k}
\|T_\alpha^*h\|^2\leq \|h\|^2
$$
for any $k\in \NN$. Consequently, $\sum_{|\alpha|=k}
\|T_\alpha^*h\|^2= \|h\|^2$ for $k\in \NN$, which proves that $h\in
\cH_c$. Therefore, we have $\cH_u\subseteq \cH_c$, where $\cH_c$ is
given by relation \eqref{hc}.

Define the subspaces
$$
\cM:=\overline{\text{\rm span}}\{T_\alpha h: \ h\in \cD_T, \alpha\in
\FF_n^+\}
$$
and $\cH_{nil}:=\cM\ominus \cH_v$, and  let $[N_1,\ldots, N_n]\in
B(\cH_{nil})^n$ be the $n$-tuple of operators given by
$N_i:=P_{\cH_{nil}}T_i|_{\cH_{nil}}$ for $i=1,\ldots,n$. Since
$$\sum_{i=1}^n N_i N_i^*\leq
 P_{\cH_{nil}}\left(\sum_{i=1}^n T_iT_i^*\right)|_{\cH_{nil}}\leq
 I_{\cH_{nil}},
 $$ we deduce that $[N_1,\ldots, N_n]\in
 [B(\cH_{nil})^n]_1^-$. Note that if $m=0$, then $\cH_{nil}=\{0\}$.
 On the other hand, since $\cM$ and $\cH_v$ are invariant subspaces
 under each operator $T_1,\ldots, T_n$, the subspace $\cH_{nil}$ is
 semi-invariant under the same operators and, consequently,
 $N_\alpha=P_{\cH_{nil}}T_\alpha|_{\cH_{nil}}$ for all $\alpha\in
 \FF_n^+$.
Note that, due to the fact that $T_\beta\cM\subseteq \cH_v$ for all
$\beta\in \FF_n^+$ with $|\beta|\geq m$, we have  $N_\beta=0$ for
$|\beta|\geq m$. Therefore, $N$ is a nilpotent row contraction of
order $\leq m$, and
$$T_i|_{\cH_v\oplus \cH_{nil}}=
 \left[\begin{matrix}V_i&*\\
0&N_i\\
\end{matrix}\right],\qquad i=1,\ldots,n,
$$
where $V_i:=T_i|_{\cH_v}:\cH_v\to \cH_v$.

Now, let $\cH_3:=\cH\ominus \cM$ and define
$W_i:=P_{\cH_3}T_i|_{\cH_3}$ for $i=1,\ldots,n$. Note that a vector
$h\in \cH$ is in $\cH_3$ if and only if $h\perp T_\alpha \Delta_Tx$
for all $x\in \cH$ and $\alpha\in \FF_n^+$, which is equivalent to
$\Delta_T T_\alpha^* h=0$ for all $\alpha\in \FF_n^+$. Consequently,
$h\in \cH_3$ if and only if
$$
(I-T_1T_1^*-\cdots-T_nT_n^*)T_\alpha^*h=0,\qquad \alpha\in \FF_n^+.
$$
Therefore, if $h\in \cH_3$, then one can prove by induction over $k\in\NN$
that
\begin{equation*}
\begin{split}
\|h\|^2&=\sum_{i=1}^n\left<T_iT_i^*h,h\right>=\sum_{|\alpha|=2}
\left<T_\alpha T_\alpha^*h,h\right> =\cdots=\sum_{|\alpha|=k}
\left<T_\alpha T_\alpha^*h,h\right>
\end{split}
\end{equation*}
for all $k\in \NN$. This shows that
$$
\cH_3\subseteq \cH_c:=\left\{h\in \cH:\ \sum_{|\alpha|=k}
\|T_\alpha^*h\|^2=\|h\|^2 \text{ for any } k=1,2,\ldots\right\}.
$$
We prove now the reverse inclusion. Since $T_i^*\cH_c\subseteq
\cH_c$ for $i=1,\ldots,n$, for any $h\in \cH_c$ and $\beta\in
\FF_n^+$, we deduce that
$$
\sum_{|\alpha|=k} \|T_\alpha^*T_\beta^* h\|^2=\|T_\beta^*h\|^2,
\qquad    k=1,2,\ldots.
$$
In particular, we have  $\left<T_\beta (I-\sum_{i=1}^n
T_iT_i^*)T_\beta^*h,h\right>=0$, whence $\Delta_T T_\beta^*h=0$ for
all $\beta\in \FF_n^+$. Therefore, $h\in \cH_3$, which completes the
proof of the fact that $\cH_3=\cH_c$, the largest  co-invariant
subspace   under
$T_1,\ldots T_n$  such that  $\left[\begin{matrix}T_1^*|_{\cH_c}\\
\vdots\\
T_n^*|_{\cH_c}
\end{matrix}\right] $ is an isometry.  This implies  that $\sum_{i=1}^n W_iW_i^*=I_{\cH_c}$.
We have  also seen that $\cH_u\subseteq \cH_c=\cH_3:=\cH\ominus \cM$ and $\cH_u\subseteq \cH_v\subseteq \cM$. Consequently, $\cH_u=\{0\}$ and $T_i$ has the representation

$$
T_i=\left[\begin{matrix}V_i&*&*\\
0&N_i&*\\
0&0&W_i
\end{matrix}\right],\qquad i=1,\ldots,n,
$$
where $V$, $N$, and $W$ are $n$-tuples of operators with the
required properties.  If $m=0$, then
$T_i$ admits the representation
$$
T_i=\left[\begin{matrix}V_i&*\\
0&W_i
\end{matrix}\right],\qquad i=1,\ldots,n,
$$
with respect to the decomposition $\cH=\cH_v\oplus \cH_c$.
The proof is complete.
\end{proof}

\begin{theorem} \label{VNW} Let $\cH_0$, $\cH_1$, and $\cH_2$  be Hilbert spaces and let $V$, $N$,
 and $W$ be $n$-tuples of operators with the following properties:
\begin{enumerate}
\item[(i)]
 $V:=[V_1,\ldots, V_n]\in
[B(\cH_0)^n]_1^-$ is  an isometry;
\item[(ii)]  $N:=[N_1,\ldots,
N_n]\in [B(\cH_1)^n]_1^-$ is  a nilpotent row contraction of order
$m\in \NN$ with $\cH_1=\{0\}$ if $m=0$;
  \item[(iii)]  $W:=[W_1,\ldots,W_n]\in [B(\cH_2)^n]_1^-$ is  a coisometry.
\end{enumerate}
     Then the following statements hold.
     \begin{enumerate}
     \item[(a)] If $m\geq 1$, then
     the characteristic function  of  any
row contraction $[T_1,\ldots,T_n]\in [B(\cH)^n]_1^-$  of the form
$$
T_i=\left[\begin{matrix}V_i&*&*\\
0&N_i&*\\
0&0&W_i
\end{matrix}\right],\qquad i=1,\ldots,n,
$$
with respect to the decomposition  $\cH=\cH_0\oplus \cH_1\oplus \cH_2$,
 is a polynomial of
degree $\leq m$.
\item[(b)] If $m=0$, then  the characteristic function  of  any
row contraction $[T_1,\ldots,T_n]\in [B(\cH)^n]_1^-$  of the form
$$
T_i=\left[\begin{matrix}V_i&*\\
0&W_i
\end{matrix}\right],\qquad i=1,\ldots,n,
$$
with respect to the decomposition  $\cH=\cH_0\oplus \cH_2$,
 is a  polynomial of
degree zero.
\end{enumerate}
\end{theorem}
\begin{proof} First, we consider the case when $m\geq 1$.
Since $V_i^*V_j=\delta_{ij}I$ for $i,j\in \{1,\ldots,n\}$, we have
$$
T_i^*T_j=
\left[\begin{matrix}\delta_{ij}I&*&*\\
*&*&*\\
*&*&*
\end{matrix}\right],\qquad i=1,\ldots,n,
$$
where $*$ stands for an unspecified entry. Consequently, we deduce
that $ \Delta_{T^*}^2=[\delta_{ij}I_\cH-T_i^*T_j]_{n\times n}=[{\bf
K}_{ij}]_{n\times n}$, where each operator entry ${\bf K}_{ij}\in
B(\cH)$ has the
form ${\bf K}_{ij}=[K_{ij}^{(pq)}]_{3\times 3}=\left[\begin{matrix}0&*&*\\
*&*&*\\
*&*&*
\end{matrix}\right]$ with
respect to the decomposition  $\cH=\cH_0\oplus \cH_1\oplus\cH_2$.
Let $\Delta_{T^*}$ have the matrix representation
$\Delta_{T^*}=[{\bf D}_{ij}]_{n\times n}$, where each entry ${\bf
D}_{ij}$  has the form $[D_{ij}^{(pq)}]_{3\times 3}$, $p,q\in
\{1,2,3\}$, with respect to the decomposition $\cH=\cH_0\oplus
\cH_1\oplus\cH_2$. Since $\Delta_{T^*}$ is a positive operator, we
must have ${\bf D}_{ii}\geq 0$ and ${\bf D}_{ji}={\bf D}_{ij}^*$ for
all $i,j\in \{1,\ldots, n\}$. This implies $D_{ii}^{(pp)}\geq 0$ for
all $i\in \{1,\ldots,n\}$ and $p\in\{1,2,3\}$, and
$D_{ji}^{(qp)}=(D_{ij}^{pq)})^*$ for all $i,j\in \{1,\ldots,n\}$ and
$p,q\in\{1,2,3\}$. Since $K_{ii}^{(11)}=0$  and
$$
K_{ii}^{(11)}=\sum_{q=1}^3
\sum_{j=1}^nD_{ij}^{(1q)}(D_{ij}^{(1q)})^*\qquad  \text{ for } \
i\in \{1,\ldots,n\},
$$
we deduce that $D_{ij}^{(1q)}=0$ for all $i,j\in \{1,\ldots,n\}$ and
$q\in \{1,2,3\}$. Therefore $\Delta_{T^*}$ has the operator matrix
representation
$$
\Delta_{T^*}=
\left[\begin{matrix}
\left[\begin{matrix}0&0&0\\
0&*&*\\
0&*&*
\end{matrix}\right] &\cdots &
\left[\begin{matrix}0&0&0\\
0&*&*\\
0&*&*
\end{matrix}\right]\\
\vdots&\vdots&\vdots\\
\left[\begin{matrix}0&0&0\\
0&*&*\\
0&*&*
\end{matrix}\right] &\cdots &
\left[\begin{matrix}0&0&0\\
0&*&*\\
0&*&*
\end{matrix}\right]
\end{matrix}\right].
$$
Now, note that
$$
\Delta_T^2=I-\sum_{i=1}^n T_iT_i^*=\left[\begin{matrix}*&*&*\\
*&*&*\\
*&*&0
\end{matrix}\right].
$$
Setting $\Delta_T=[\Lambda_{pq}]_{3\times 3}$ and taking into
account that $\Delta_T\geq 0$, we deduce that $\Lambda_{pp}\geq 0$
and $\Lambda_{qp}=\Lambda_{pq}^*$ for $p,q\in \{1,2,3\}$. Since
$$
\Delta_T^2= \left[\begin{matrix}*&*&*\\
*&*&*\\
*&*&\Lambda_{13}^*\Lambda_{13}+\Lambda_{23}^*\Lambda_{23}+\Lambda_{33}^2
\end{matrix}\right],
$$
we must have
$\Lambda_{13}^*\Lambda_{13}+\Lambda_{23}^*\Lambda_{23}+\Lambda_{33}^2=0$,
which implies $\Lambda_{13}=\Lambda_{23}=\Lambda_{33}=0$. Therefore,
$\Delta_T$ has the form
$$
\Delta_T=\left[\begin{matrix}*&*&0\\
*&*&0\\
0&0&0
\end{matrix}\right]
$$
 with
respect to the decomposition  $\cH=\cH_0\oplus \cH_1\oplus\cH_2$.
Since
$$
T_\beta=\left[\begin{matrix}V_\beta&*&*\\
0&0&*\\
0&0&W_\beta
\end{matrix}\right]
$$
for all  $\beta\in \FF_n^+$ with $|\beta|\geq m\geq 1$, we deduce that
\begin{equation*}
\begin{split}
\Delta_TT_\beta^*P_i\Delta_{T^*} (\oplus_{i=1}^n h_i)&= \sum_{i=1}^n
\left[\begin{matrix}*&*&0\\
*&*&0\\
0&0&0
\end{matrix}\right]\left[\begin{matrix}V_\beta^*&0&0\\
*&0&0\\
*&*&W_\beta^*
\end{matrix}\right]\left[\begin{matrix}0&0&0\\
0&*&*\\
0&*&*
\end{matrix}\right]h_i \\
&=\sum_{i=1}^n \left[\begin{matrix}*&0&0\\
*&0&0\\
0&0&0
\end{matrix}\right]
\left[\begin{matrix}0&0&0\\
0&*&*\\
0&*&*
\end{matrix}\right]h_i=0
\end{split}
\end{equation*}
for any $\oplus_{i=1}^n h_i\in \cH^{(n)}$. Hence,
$\Delta_TT_\beta^*P_i\Delta_{T^*}=0$ for all $\beta\in \FF_n^+$ with
$|\beta|\geq m\geq 1$, which shows that the characteristic function
$\Theta_T$ is a polynomial of degree $\leq m$.

Now, we consider the case when $m=0$. Similar considerations as above reveal that $\Delta_{T^*}$ and $\Delta_T$ have the forms
$$
\Delta_{T^*}=
\left(\begin{matrix}
\left[\begin{matrix}0&0\\
0&*
\end{matrix}\right] &\cdots &
\left[\begin{matrix}0&0\\
0&*
\end{matrix}\right]\\
\vdots&\vdots&\vdots\\
\left[\begin{matrix}0&0\\
0&*
\end{matrix}\right] &\cdots &
\left[\begin{matrix}0&0\\
0&*
\end{matrix}\right]
\end{matrix}\right)
$$
and $
\Delta_T=\left[\begin{matrix}*&0\\
0&0
\end{matrix}\right]
$
 with
respect to the decomposition  $\cH=\cH_0\oplus\cH_2$.
Since
$
T_\beta=\left[\begin{matrix}V_\beta&*\\
0&W_\beta
\end{matrix}\right]$
for all  $\beta\in \FF_n^+$, we have
\begin{equation*}
\begin{split}
\Delta_TT_\beta^*P_i\Delta_{T^*} (\oplus_{i=1}^n h_i)&= \sum_{i=1}^n
\left[\begin{matrix}*&0\\
0&0
\end{matrix}\right]\left[\begin{matrix}V_\beta^*&0\\
*&W_\beta^*
\end{matrix}\right]\left[\begin{matrix}0&0\\
0&*
\end{matrix}\right]h_i \\
&=\sum_{i=1}^n \left[\begin{matrix}*&0\\
0&0
\end{matrix}\right]
\left[\begin{matrix}0&0\\
0&*
\end{matrix}\right]h_i=0
\end{split}
\end{equation*}
for any $\oplus_{i=1}^n h_i\in \cH^{(n)}$ and $\beta\in \FF_n^+$. Hence, we deduce that the characteristic function $\Theta_T$ is a constant, i.e., $\Theta_T=\Theta_T(0)$.
The proof is
complete.
\end{proof}

Combining Theorem \ref{structure} and Theorem \ref{VNW}, we obtain
the following characterization for   row contractions with
polynomial characteristic functions.

\begin{theorem}\label{main-structure} Let $T:=[T_1,\ldots, T_n]\in [B(\cH)^n]_1^-$ be a
 row contraction. Then the characteristic function $\Theta_T$
is a noncommutative polynomial of degree $m\in \NN$ if and only if
there exist subspaces
 $\cH_0$, $\cH_{1}$, and $\cH_2$ of $\cH$  such that
$\cH=\cH_0\oplus \cH_{1}\oplus \cH_2$ and each $T_i$ admits a
representation
$$
T_i=\left[\begin{matrix}V_i&*&*\\
0&N_i&*\\
0&0&W_i
\end{matrix}\right],\qquad i=1,\ldots,n,
$$
where $V:=[V_1,\ldots, V_n]\in [B(\cH_0)^n]_1^-$ is a pure isometry,
$N:=[N_1,\ldots, N_n]\in [B(\cH_{1})^n]_1^-$ is a nilpotent row
contraction of order $m$, and $W:=[W_1,\ldots,W_n]\in
[B(\cH_2)^n]_1^-$ is a coisometry. Moreover, the degree of
$\Theta_T$ is the smallest possible order of $N$  in the
representation of $T$.
\end{theorem}

In general, a row contraction has many representations in  upper
triangular form. The next result shows that, in a certain sense, the
representation provided  by Theorem \ref{structure} is unique.

\begin{proposition} Let $T:=[T_1,\ldots, T_n]\in [B(\cH)^n]_1^-$ be a
 row contraction such that the characteristic  function
$\Theta_T$ is a noncommutative polynomial of degree $m\in \NN$.
 Let $T:=[T_1,\ldots, T_n]$ have a representation
$$
T_i=\left[\begin{matrix}V_i'&*&*\\
0&N_i'&*\\
0&0&W_i'
\end{matrix}\right],\qquad i=1,\ldots,n,
$$
with respect to a decomposition $\cH=\cH_0'\oplus \cH_1'\oplus
\cH_2'$, where $[V_1',\ldots, V_n']\in [B(\cH_0')^n]_1^-$ is an
isometry,
 $[N_1',\ldots, N_n']\in [B(\cH_{1}')^n]_1^-$ is a nilpotent row
contraction of order $m$, and $[W_1',\ldots,W_n']\in
[B(\cH_2')^n]_1^-$ is a coisometry.

Then the  upper triangular representation of $T$ given by Theorem
\ref{structure} has the  following properties: $\cH_v\subseteq
\cH_0'$, $\cH_c\supseteq \cH_2'$. Moreover,
\begin{equation*}\begin{split}
\cH_v&=\overline{\text{\rm span}}\{T_\beta h:\ h\in \cD_T,
|\beta|\geq m\},\\
 \cH_c&=\{h\in \cH: \ \sum_{|\alpha|=k} \|T_\alpha^*
h\|^2=\|h\|^2
\text{ for all } k\in \NN\}, \quad \text{ and }\\
\cH_{nil}&=\cH\ominus (\cH_v\oplus \cH_c).
\end{split}
\end{equation*}

\end{proposition}
\begin{proof} As  in the proof of Theorem \ref{VNW},  we deduce that $\Delta_T$ has the form
$$
\Delta_T=\left[\begin{matrix}*&*&0\\
*&*&0\\
0&0&0
\end{matrix}\right]
$$
 with
respect to the decomposition  $\cH=\cH_0'\oplus \cH_1'\oplus\cH_2'$
and
$$
T_\alpha \Delta_T=
\left[\begin{matrix}V_\beta'&*&*\\
0&0&*\\
0&0&W_\beta'
\end{matrix}\right]\left[\begin{matrix}*&*&0\\
*&*&0\\
0&0&0
\end{matrix}\right]=\left[\begin{matrix}*&*&0\\
0&0&0\\
0&0&0
\end{matrix}\right]
$$
for any $\alpha\in \FF_n^+$ with $|\alpha|\geq m$. Consequently, if
$h=h_0\oplus h_1\oplus h_1$, where $h_j\in \cH_j'$ for
$j=0,1,2$, then $T_\alpha\Delta h=\left[\begin{matrix}* \\
0 \\
0
\end{matrix}\right]\in \cH_0'$ for
$|\alpha|\geq m$. Hence $\cH_v\subset \cH_0'$. Note that the
inclusion $\cH_2'\subseteq \cH_c$ is true due to the fact that
$\cH_c$ is the largest invariant subspace   under
$T_1^*,\ldots T_n^*$  such that  $\left[\begin{matrix}T_1^*|_{\cH_c}\\
\vdots\\
T_n^*|_{\cH_c}
\end{matrix}\right] $ is an isometry.
The last part of the proposition follows from the proof of Theorem
\ref{structure}.
\end{proof}

In what follows, we call the  upper triangular representation of $T$
given by Theorem \ref{structure}   canonical.

We recall that a row contraction $T:=[T_1,\ldots, T_n]\in
[B(\cH)^n]_1^-$ is called  {\it completely non-unitary}  (c.n.u.) if
there is no  nonzero subspace    $\cM\subseteq \cH$  reducing under
$T_1,\ldots,T_n$ such that $[T_1|_\cM,\ldots, T_n|_\cM]$ is a
unitary operator from $\cM^{(m)}$ to $\cM$.

Using Theorem \ref{structure}, one can easily deduce the following
\begin{corollary}
\label{structure2} Let $T:=[T_1,\ldots, T_n]\in [B(\cH)^n]_1^-$ be a
c.n.u. row contraction.  Then the characteristic function $\Theta_T$
is a noncommutative polynomial of degree $m\in \NN$ if and only if
there exist subspaces
 $\cH_0$, $\cH_1$, and $\cH_2$ of $\cH$  such that
$\cH=\cH_0\oplus \cH_1\oplus \cH_2$ and each $T_i$ admits a
representation
$$
T_i=\left[\begin{matrix}V_i&*&*\\
0&N_i&*\\
0&0&C_i
\end{matrix}\right],\qquad i=1,\ldots,n,
$$
where $V:=[V_1,\ldots, V_n]\in [B(\cH_0)^n]_1^-$ is a  pure row
isometry, $N:=[N_1,\ldots, N_n]\in [B(\cH_1)^n]_1^-$ is a nilpotent
row contraction of order $m$, and $C:=[C_1,\ldots,C_n]\in
[B(\cH_2)^n]_1^-$ is a c.n.u. coisometry.
\end{corollary}
We remark that there is  a canonical upper triangular representation
for c.n.u. row contractions,   namely, the one provided by Theorem
\ref{structure}.

\bigskip

\section{Unitary invariants on the unit ball of $B(\cH)^n$}

In general, a row contraction has many representations in  upper
triangular form.  The next result gives another  reason why we will
focus on the canonical upper triangular representations of row contractions  with
polynomial characteristic functions.

\begin{proposition} Let $T:=[T_1,\ldots, T_n]\in [B(\cH)^n]_1^-$
 and $T':=[T_1',\ldots, T_n']\in [B(\cH')^n]_1^-$ be
  row contractions with  polynomial characteristic functions, and let
  $$
T_i=\left[\begin{matrix}V_i&*&*\\
0&N_i&*\\
0&0&W_i
\end{matrix}\right]\quad \text{ and } \quad T_i'=\left[\begin{matrix}V_i'&*&*\\
0&N_i'&*\\
0&0&W_i'
\end{matrix}\right]
$$
be their   canonical representations on $\cH=\cH_v\oplus
\cH_{nil}\oplus \cH_c$ and $\cH'=\cH_v'\oplus \cH_{nil}'\oplus
\cH_c'$, respectively. If $U:\cH\to \cH'$ is a unitary operator such
that $UT_i=T_i'U$ for all $i=1,\ldots,n$, then
$$U(\cH_v)=\cH_v', \quad
U(\cH_{nil})=\cH_{nil}', \quad   U(\cH_c)=\cH_c', $$
 and the diagonal entries of $T$ and $T'$ are unitarily equivalent,
 i.e.,
$$(U|_{\cH_v})V_i=V_i'(U|_{\cH_v}),\quad
(U|_{\cH_{nil}})N_i=N_i'(U|_{\cH_{nil}}),\quad
(U|_{\cH_c})W_i=W_i'(U|_{\cH_c})
$$
for all $i=1,\ldots,n$.
Moreover, if $T:=[T_1,\ldots, T_n]$ has a representation
$$
T_i=\left[\begin{matrix}A_i&*&*\\
0&B_i&*\\
0&0&C_i
\end{matrix}\right],\qquad i=1,\ldots,n,
$$
with respect to a decomposition $\cH=\cH_0\oplus \cH_1\oplus \cH_2$,
where $[A_1,\ldots, A_n]\in [B(\cH_0)^n]_1^-$ is a pure isometry,
 $[B_1,\ldots, B_n]\in [B(\cH_{1})^n]_1^-$ is a nilpotent row
contraction of order $m\in \NN$, and $[C_1,\ldots,C_n]\in
[B(\cH_2)^n]_1^-$ is a coisometry, then the diagonal entries of \
$T$ are not, in general, unitarily equivalent  with  those
corresponding  to  the canonical representation of $T$.
\end{proposition}
\begin{proof} According to Section 1,
we have
\begin{equation*}\begin{split}
\cH_v&=\overline{\text{\rm span}}\{T_\beta h:\ h\in \cD_T,
|\beta|\geq m\},\\
 \cH_c&=\{h\in \cH: \ \sum_{|\alpha|=k} \|T_\alpha^*
h\|^2=\|h\|^2
\text{ for all } k\in \NN\},\\
\cH_{nil}&=\cH\ominus (\cH_v\oplus \cH_c),
\end{split}
\end{equation*}
and similar formulas hold for $\cH_v', \cH_c'$ and  $\cH_{nil}'$, respectively. If
$U:\cH\to \cH'$ is a unitary operator such that $UT_i=T_i'U$ for
$i=1,\ldots,n$, then   $U\Delta_T=\Delta_{T'}U$ and
$U(\cH_v)=\cH_v'$, $ U(\cH_{nil})=\cH_{nil}'$, and
$U(\cH_c)=\cH_c'$. Now,  it is easy to see that the diagonal entries
of $T$ and $T'$ are unitarily equivalent.

To prove the last part of the proposition,  let  $\cN$ be a
separable Hilbert space and let $C_i\in B(\cN)$ be such that
$C=[C_1,\ldots,C_n]$ is a coisometry. Fix $m\geq 1$ and denote by
$\cP_{m-1}$ the subspace of all polynomials of degree $\leq m-1$ in
the full Fock space $F^2(H_n)$, i.e. $\cP_{m-1}:=\text{\rm
span}\{e_\alpha: \ \alpha\in \FF_n^+, |\alpha|\leq m-1\}$. Let
$T:=[T_1,\ldots,T_n]$ be defined by
$$
T_i=\left[\begin{matrix}S_i&0&0\\
0&P_{\cP_{m-1}}S_i|_{\cP_{m-1}}&0\\
0&0&C_i
\end{matrix}\right],\qquad i=1,\ldots,n,
$$
with respect to the decomposition $\cH:=F^2(H_n)\oplus
\cP_{m-1}\oplus \cN$, where $S_1,\ldots, S_n$ are the left creation
operators on $F^2(H_n)$. According to Theorem \ref{structure}, the
canonical  decomposition of $T_i$ is
$$
T_i=\left[\begin{matrix}V_i&*&*\\
0&N_i&*\\
0&0&C_i
\end{matrix}\right],\qquad i=1,\ldots,n,
$$
with respect to the decomposition $\cH=\cH_v\oplus \cH_{nil}\oplus
\cH_{c}$, where
$$
\cH_v:=[F^2(H_n)\ominus\cP_{m-1}]\oplus 0\oplus 0,\quad
\cH_{nil}:=\cP_{m-1}\oplus \cP_{m-1}\oplus 0,\quad \text{ and }\quad
\cH_c:=0\oplus 0 \oplus \cN,
$$
 the operators $V_i\in B(F^2(H_n)\ominus\cP_{m-1})$, $N_i\in
B(\cP_{m-1}\oplus \cP_{m-1})$, and $C_i\in B(\cL)$ are defined by
$$
V_i:=S_i|_{F^2(H_n)\ominus\cP_{m-1}},\quad N_i:=\left[\begin{matrix}P_{\cP_{m-1}}S_i|_{\cP_{m-1}}&0\\
0&P_{\cP_{m-1}}S_i|_{\cP_{m-1}}
\end{matrix}\right],\quad \text{ and }\quad W_i:=C_i
$$
for any $i=1,\ldots,n$. We remark that the pure isometries
$[S_1,\ldots,S_n]$ and $[V_1,\ldots,V_n]$ are not unitarily
equivalent, when $n\geq 2$,  since they have  the multiplicity $1$
and $n^m$, respectively. Note also that the nilpotent row
contractions $[P_{\cP_{m-1}}S_1|_{\cP_{m-1}}, \ldots,
P_{\cP_{m-1}}S_n|_{\cP_{m-1}}]$ and $[N_1, \ldots, N_n]$ are not
unitarily equivalent, in spite of having the same order $m$. The
proof is complete.
\end{proof}

We need to recall from \cite{Po-poisson} that the noncommutative
Poisson kernel associated with  a row contraction $T:=[T_1,\ldots, T_n]\in  [B(\cH)^n]_1^-$ is the operator
$
K_{T} :\cH\to F^2(H_n)\otimes \overline{\Delta_{T}\cH}
$
defined by

\begin{equation*}
K_{T}h:= \sum_{k=0}^\infty \sum_{|\alpha|=k} e_\alpha\otimes
\Delta_{T} T_\alpha^*h,\quad h\in \cH.
\end{equation*}
The operator $K_{rT}$ is an  isometry if $0<r<1$, and
$$
K_T^*K_T=I-
\text{\rm SOT-}\lim_{k\to\infty} \sum_{|\alpha|=k} T_\alpha T_\alpha^*.
$$
 The connection between the characteristic function and the Poisson kernel of a row contraction is given by the formula
$ I-\Theta_T\Theta_T^*=K_TK_T^*$ (see \cite{Po-varieties}).

 Let $\NN_\infty:=\NN\cup \{\infty\}$ and
define the map
$$\Gamma:[B(\cH)^n]_1^-\to \NN_\infty\times \NN_\infty\times
\NN_\infty, \qquad   \Gamma(T):=(p,m,q),
$$
by setting $m:=\deg (\Theta_T)$, \  $q:=\dim (\ker K_T)$, and
 $$p:=\begin{cases} \dim (\cD_m\ominus \cD_{m+1})& \quad \text{ if  } m\in \NN\\
 \dim \overline{\Delta_T\cH}& \quad \text{ if }
m=\infty,
\end{cases}$$
  where
 $
 \cD_m:=\overline{\text{\rm span}}\{T_\beta \Delta_Th:\ h\in \cH, |\beta|\geq
m\}, $
 $\Theta_T$ is the characteristic function, $K_T$ is the
noncommutative Poisson kernel, and $\Delta_T$ is the defect operator
associated with $T\in [B(\cH)^n]_1^-$. One can easily show that the
map $\Gamma$ is a unitary invariant for row contractions, i.e., if
$T \in [B(\cH)^n]_1^-$
 and $T' \in [B(\cH')^n]_1^-$ are unitarily
 equivalent, then $\Gamma(T)=\Gamma(T')$.

The next result shows that the map $\Gamma$    detects the pure row
isometries in the closed  unit ball of $B(\cH)^n$ and completely
classify them up to a unitary equivalence.

\begin{theorem} \label{invariant1} Let $T:=[T_1,\ldots, T_n]\in [B(\cH)^n]_1^-$ be a
  row contraction. Then the following statements hold:
  \begin{enumerate}
  \item[(i)] $T$ is a pure isometry if and only if \
   $\Gamma(T)\in \NN_\infty \times\{0\}\times\{0\}$.
   \item[(ii)]  If \ $T,T'\in
  [B(\cH)^n]_1^-$ and $\Gamma(T)=\Gamma(T')=(p,0,0)$ for some $p\in
  \NN_\infty$, then $T$ is unitarily equivalent to $T'$ and $p=\rank \Delta_T=\rank \Delta_{T'}$.
  \end{enumerate}
\end{theorem}
\begin{proof} First, we assume that  $T$ is a pure isometry. According to the Wold decomposition for isometries with orthogonal subspaces \cite{Po-isometric}, $T$ is unitarily equivalent to
$(S_1\otimes
  I_\cK,\ldots, S_n\otimes I_\cK)$ for some Hilbert space $\cK$. Therefore, without loss of generality
    we can  assume that $T=[S_1\otimes
  I_\cK,\ldots, S_n\otimes I_\cK]$. In this case, we have  $\Delta_T=P_\CC\otimes I_\cK$ and
  $\Delta_{T^*}=0$.  Consequently, we deduce that $\cD_T=1\otimes
  \cK$, $\cD_{T^*}=\{0\}$, and $\Theta_T=0\in B(\{0\}, \cK)$. Hence, $\deg (\Theta_T)=0$. Note that
  $$
   \left[\overline{\text{\rm span}}\{T_\beta \Delta_Th:\ h\in \cH, |\beta|\geq
m\}\ominus \overline{\text{\rm span}}\{T_\beta \Delta_Th:\ h\in \cH,
|\beta|\geq m+1\}\right]=1\otimes \cK
$$
and  $p=\dim \cK=\rank \Delta_T$. On the other hand, since
$$ \ker K_T=\{h\in \cH: \ \sum_{|\alpha|=k} \|T_\alpha^*
h\|^2=\|h\|^2 \text{ for all } k\in \NN\}=\cK_c$$ and
$T=[S_1\otimes
  I_\cK,\ldots, S_n\otimes I_\cK]$ is a completely non-coisometric row contraction,
   we have $\ker K_T=\cH_c=\{0\}$ and, therefore, $\dim \ker K_T=0$. Summing up, we deduce that
   $\Gamma(T)\in \NN_\infty \times\{0\}\times\{0\}$.

Conversely, assume that $T$ is a row contraction with
  $\Gamma(T)\in \NN_\infty \times\{0\}\times\{0\}$. Then  we have $\ker K_T=\cH_c=\{0\}$.
  According to Theorem \ref{structure},
$T_i$ admits the representation
$$
T_i=\left[\begin{matrix}V_i&*\\
0&N_i
\end{matrix}\right],\qquad i=1,\ldots,n,
$$
with respect to the decomposition $\cH=\cH_v\oplus \cH_{nil}$. On the other hand,
  since $\deg(\Theta_T)=0$, we must
   have  $\cH_{nil}=\{0\}$ and $T_i=V_i$ for  $i=1,\ldots,n$.
   Therefore, $T=[V_1,\ldots,V_n]$ is a pure isometry on $\cH$ and, using   the Wold decomposition
    for isometries with orthogonal subspaces, we deduce that
 $$
  p=\dim \left[\overline{\text{\rm span}}\{T_\beta h:\ h\in \cD_T, \beta\in \FF_n^+
\}\ominus \overline{\text{\rm span}}\{T_\beta h:\ h\in \cD_T,
|\beta|\geq 1\}\right]
$$
is the dimension of the wandering subspace for $T=[V_1,\ldots,V_n]$. Hence, $T=[V_1,\ldots,V_n]$ is
 unitarily equivalent to
$[S_1\otimes
  I_\cK,\ldots, S_n\otimes I_\cK]$ for some Hilbert space $\cK$ with $\dim \cK=p$,
  where $S_1,\ldots, S_n$  are the left creation operators on the full Fock space
  $F^2(H_n)$. Therefore, part (i) holds.

   To prove part (ii), assume that \ $T,T'\in
  [B(\cH)^n]_1^-$ and $\Gamma(T)=\Gamma(T')=(p,0,0)$ for some $p\in
  \NN_\infty$. Due to the first part of the proof, we deduce that $T$
  and $T'$ are pure row contractions with the property that the
  dimensions of their wandering subspaces are equal to $p=\rank
  \Delta_T=\rank \Delta_{T'}$. Consequently, using the Wold
  decomposition,
   we conclude that the pure row isometries $T$
  and $T'$ are unitarily equivalent. The proof is complete.
\end{proof}

We remark that, due to Theorem \ref{invariant1} and the  model
theory for row contraction \cite{Po-charact}, if  $q=0$ and $m=0$ or
$q=0$ and  $m=\infty$, then $p$ represents the multiplicity of the
$n$-tuple $(S_1,\ldots,S_n)$ of left creation operators in the
operator model of $T=(T_1,\ldots,T_n)$.

\begin{corollary} Let $T:=(T_1,\ldots, T_n)\in [B(\cH)^n]_1^-$ be a
  row contraction and $S_1,\ldots, S_n$  be the left creation operators on the full Fock space $F^2(H_n)$.
 Then $T$ is  unitarily equivalent to $(S_1\otimes
  I_\cK,\ldots, S_n\otimes I_\cK)$ for some Hilbert space $\cK$  if and only
  if
  $$\Gamma(T)=(\dim \cK, 0, 0).
  $$
   In this case, $\rank \Delta_T=\dim \cK$.
\end{corollary}

Let $\Phi:[B(\cH)^n]_1\to B(\cH)\bar\otimes B(\cK_1, \cK_2)$ and
$\Phi':[B(\cH)^n]_1\to B(\cH)\bar\otimes B(\cK_1', \cK_2')$ be two
free holomorphic functions. We say that $\Phi$ and $\Phi'$ coincide
if there are two unitary operators $\tau_j\in B(\cK_j,\cK_j')$,
$j=1,2$, such that
$$
\Phi'(X)(I_\cH\otimes \tau_1)=(I_\cH\otimes \tau_2)\Phi(X),\qquad
X\in [B(\cH)^n]_1.
$$

Now, we can prove the following classification result.

\begin{theorem}  Let $T:=[T_1,\ldots, T_n]\in [B(\cH)^n]_1^-$ be a
  row contraction.
  Then  the following statements hold:
  \begin{enumerate}
  \item[(i)] $T$ is   a pure row contraction  with  polynomial characteristic function
  if and only if
  $$\Gamma(T)\in \NN_\infty\times \NN\times\{0\}.
  $$
   In this case, $T_i$ has
  the canonical
  form $
T_i=\left[\begin{matrix}V_i&* \\
0&N_i
\end{matrix}\right]$,
where $V:=[V_1,\ldots, V_n]\in [B(\cH_v)^n]_1^-$ is a pure isometry
and  $N:=[N_1,\ldots, N_n]\in [B(\cH_{nil})^n]_1^-$ is a nilpotent
row contraction.
\item[(ii)] the map $T\mapsto  (\Gamma(T), \Theta_T)$  detects the pure
row contractions with polynomial characteristic functions  and
completely classify them.

\end{enumerate}
\end{theorem}
\begin{proof}
Assume that $T$ is a pure row contraction with polynomial
characteristic function. Let
 $$
T_i=\left[\begin{matrix}V_i&*&*\\
0&N_i&*\\
0&0&W_i
\end{matrix}\right]
$$
be the  canonical upper triangular representation on
$\cH=\cH_v\oplus \cH_{nil}\oplus \cH_c$, provided by Theorem
\ref{structure}.  If $h\in \cH_c$, then $T_\alpha^*(0\oplus 0\oplus
\oplus h)=0\oplus 0\oplus   W_\alpha^* h$. Consequently, we have
$$
\|h\|^2=\sum_{|\alpha|=k} \|W_\alpha^* h\|^2=\sum_{|\alpha|=k}
\|T_\alpha^*(0\oplus 0  \oplus h)\|^2, \qquad k\in\NN.
$$
Since $T$ is a pure row contraction, we deduce that $h=0$, which
shows that  $\cH_c=\{0\}$.  Therefore, $\Gamma(T)\in
\NN_\infty\times \NN\times\{0\}$ and $T_i$ has the form $
T_i=\left[\begin{matrix}V_i&* \\
0&N_i
\end{matrix}\right]
$ with respect to the decomposition $\cH=\cH_v\oplus \cH_{nil}$,
where $V:=[V_1,\ldots, V_n]\in [B(\cH_v)^n]_1^-$ is a pure isometry
and  $N:=[N_1,\ldots, N_n]\in [B(\cH_{nil})^n]_1^-$ is a nilpotent
row contraction.

 Conversely, assume that $T$ is a row contraction with  $\Gamma(T)\in \NN_\infty\times \NN\times\{0\}$.
   Hence, $\dim(\ker K_T)=0$ and
 $\cH_c=\{0\}$. According to Theorem  \ref{structure},
 $T_i$ has the form $
T_i=\left[\begin{matrix}V_i& * \\
0&N_i
\end{matrix}\right]$. Assuming that  $N$ is a nilpotent $n$-tuple of  order
$m$, we deduce that there exist operators $X_{(\alpha)}\in
B(\cH_{nil}, \cH_v)$ such that
 \begin{equation}\label{MA}
T_\alpha=\left[\begin{matrix}V_\alpha& X_{(\alpha)} \\
0&0
\end{matrix}\right]
\end{equation}
for all $\alpha\in \FF_n^+$ with $|\alpha|=m$. Since $[T_1,\ldots,
T_n]$ is a row contraction, so is  the row operator $[T_\alpha:\
|\alpha|=k]$ for any $k\geq 1$. In particular, we have
$$
\sum_{|\alpha|=m} \|T_\alpha^*(x\oplus 0 )\|^2=\sum_{|\alpha|=m}
\|V_\alpha^*x \|^2 + \sum_{|\alpha|=m} \|X_{(\alpha)}^*x\|^2\leq
\|x\|^2
$$
for any, $x\in \cH_v$.  Consequently,  the row operator
$[X_{(\alpha)}:\ |\alpha|=m ]$ is a contraction.
Let $\alpha_1,\ldots, \alpha_k\in \FF_n^+$ be such that
$|\alpha_1|=\cdots =|\alpha_k|=m$, and note that, due to relation
\eqref{MA},
$$
T_{\alpha_1}\cdots
T_{\alpha_k}=\left[\begin{matrix}V_{\alpha_1}\cdots V_{\alpha_k}&
V_{\alpha_1}\cdots V_{\alpha_{k-1}}X_{(\alpha_k)}
 \\
0&0
\end{matrix}\right].
$$
 Since $[X_{(\alpha)}:\
|\alpha|=m ]$ is a contraction, we have
\begin{equation*}
\begin{split}
\sum_{\alpha_1,\ldots, \alpha_k\in \FF_n^+\atop
|\alpha_1|=\cdots=|\alpha_k|=m}\|T_{\alpha_1\cdots
\alpha_k}^*(x\oplus y)\|^2 &= \sum_{\alpha_1,\ldots, \alpha_k\in
\FF_n^+\atop|\alpha_1|=\cdots=|\alpha_k|=m}\|V_{\alpha_1\cdots
\alpha_k}^*x\|^2 +\sum_{\alpha_1,\ldots, \alpha_k\in
\FF_n^+\atop|\alpha_1|=\cdots=|\alpha_k|=m}\|X_{\alpha_k}^*V_{\alpha_1\cdots
\alpha_{k-1}}^*x\|^2\\
&\leq \sum_{\gamma\in \FF_n^+,|\gamma|=mk}\|V_{\gamma}^*x\|^2
+\sum_{\gamma\in \FF_n^+, |\gamma|=m(k-1)}\|V_{\gamma}^*x\|^2
\end{split}
\end{equation*}
for any $x\oplus y\in \cH_v\oplus \cH_{nil}$. Taking into account
that $[V_1,\ldots, V_n]$ is a pure isometry, we have
$$\lim_{k\to\infty}\sum_{\gamma\in \FF_n^+, |\gamma|=k}\|V_{\gamma}^*x\|^2=0,\qquad
x\in \cH_v.
$$
Hence, and using the inequalities above, we conclude that
\begin{equation}
\label{Li} \lim_{k\to\infty} \sum_{\gamma\in \FF_n^+,|\gamma|=
mk}\|T_\gamma^*(x\oplus y)\|^2 =0,\qquad x\oplus y\in \cH_v\oplus
\cH_{nil}.
\end{equation}
If $q\geq \NN$, then $q=mk_q+p_q$ for unique $k_q\in \NN$ and
$p_q\in \{0,1,\ldots, m-1\}$. Using the fact that $[T_\gamma: \
|\gamma|=p_q]$ is a row contraction, we have
\begin{equation*}
\begin{split}
\sum_{|\alpha|= q}\|T_\gamma^*(x\oplus y)\|^2 &=
 \sum_{|\gamma|=p_q, |\sigma|=mk_q}\|T_\gamma^* T_\sigma^*(x\oplus
 y)\|^2\\
 &\leq  \sum_{  |\sigma|=mk_q}\|  T_\sigma^*(x\oplus
 y)\|^2.
\end{split}
\end{equation*}
Hence, and using \eqref{Li}, we deduce that $\lim_{p\to\infty}
\sum_{|\alpha|= q}\|T_\gamma^*(x\oplus y)\|^2 =0$, which proves that
$[T_1,\ldots, T_n]$ is a pure row contraction. The proof of part (i)
complete.

To prove part (ii), let $T,T'\in
  [B(\cH)^n]_1^-$ be row contractions. Using the result from part (i) and Theorem 5.4 from \cite{Po-charact}, we
   deduce  that $T$ and
  $T'$ are unitarily equivalent pure row contractions with
  polynomial characteristic functions if and only if $\Gamma(T)$ and
  $\Gamma(T')$ are in $ \NN_\infty\times \NN\times\{0\}$, and the
  characteristic functions $\Theta_T$ and $\Theta_{T'}$ coincide.
  This completes the proof.
\end{proof}

Using Theorem \ref{main-structure}, we can easily deduce the
following
\begin{proposition}  Let $T:=[T_1,\ldots, T_n]\in [B(\cH)^n]_1^-$ be a
  row contraction with polynomial characteristic function. Then the
  following statements hold.
  \begin{enumerate}
\item[(i)] $T_i$ has the form
  $
T_i=\left[\begin{matrix}N_i&* \\
0&W_i
\end{matrix}\right]$  if and only if  \ $\Gamma(T)\in \{0\}\times \NN\times \NN_\infty$.
\item[(ii)] $T_i$ has the form $[N_i]$ if and only if  \ $\Gamma(T)\in \{0\}\times \NN\times \{0\}$.
\item[(iii)] $T_i$ has the form $[W_i]$ if and only if  \ $\Gamma(T)\in \{0\}\times \{0\}\times \NN_\infty$.
 \end{enumerate}
\end{proposition}

\begin{corollary} \label{cnc-ga}Let $T:=[T_1,\ldots, T_n]\in [B(\cH)^n]_1^-$ be a
 row contraction. Then  $T$ is c.n.c. if and only if $\Gamma(T)\in \NN_\infty\times \NN_\infty\times
 \{0\}$. In this case, the characteristic function $\Theta_T$
is a noncommutative polynomial of degree $m\in \NN$ if and only if
there exist subspaces
 $\cH_v$ and  $\cH_{nil}$ of $\cH$  such that
$\cH=\cH_v\oplus \cH_{nil}$ and each $T_i$ admits a representation
$$
T_i=\left[\begin{matrix}V_i&*\\
0&N_i
\end{matrix}\right],\qquad i=1,\ldots,n,
$$
where $V:=[V_1,\ldots, V_n]\in [B(\cH_v)^n]_1^-$ is a pure row
isometry and $N=[N_1,\ldots, N_n]\in [B(\cH_{nil})^n]_1^-$ is a
nilpotent row contraction of order $m$. Moreover, the degree of
$\Theta_T$ is the smallest  possible order of $N$ in the
representation of $T$.
\end{corollary}
\begin{proof}
Since $T$ is c.n.c. row contraction, we must have $\cH_c=\{0\}$.
Applying Theorem \ref{main-structure}, the result follows.
\end{proof}

We remark that the map $T\mapsto  (\Gamma(T), \Theta_T)$  detects the c.n.c.
row contractions  and
completely classify them. Indeed, Corollary \ref{cnc-ga} above  and Theorem 5.4 from \cite{Po-charact}, imply that
    $T$ and
  $T'$ are unitarily equivalent c.n.c. row contractions  if and only if $\Gamma(T)$ and
  $\Gamma(T')$ are in $ \NN_\infty\times \NN_\infty\times\{0\}$ and the
  characteristic functions $\Theta_T$ and $\Theta_{T'}$ coincide.

The next result is a characterization of row contractions with
constant characteristic function.
\begin{theorem} \label{charact-const} Let $T:=[T_1,\ldots, T_n]\in [B(\cH)^n]_1^-$ be a
 row contraction.  Then the following statements are equivalent:
 \begin{enumerate}
\item[(i)] the characteristic function $\Theta_T$ is a constant,   i.e,
$\Theta_T=\Theta_T(0)$;
\item[(ii)] $\Gamma(T)\in \NN_\infty\times \{0\}\times \NN_\infty$;
\item[(iii)] $T$
   admits  the canonical representation
$$
T_i=\left[\begin{matrix}V_i&*\\
0&W_i
\end{matrix}\right],\qquad i=1,\ldots,n,
$$
where $V:=[V_1,\ldots, V_n]\in [B(\cH_v)^n]_1^-$ is a pure  isometry
and $W:=[W_1,\ldots,W_n]\in [B(\cH_c)^n]_1^-$ is a coisometry.
\end{enumerate}
If, in addition, $T$ is  c.n.u, then  $\Theta_T$ is constant if and only if $T$ has the representation above where
 $V$ is a pure isometry and $W$ is a c.n.u. coisometry.
\end{theorem}
\begin{proof} Using  Theorem \ref{main-structure}, Corollary
\ref{structure2}, and the definition of the map $\Gamma$, the result
follows.
\end{proof}

\bigskip

\section{ The automorphism group
$\text{\rm Aut}(B(\cH)^n_1)$ and unitary projective representation}

The theory of noncommutative characteristic functions for row
contractions \cite{Po-charact} was used  in \cite{Po-automorphism}
to determine the group $\text{\rm Aut}(B(\cH)^n_1)$  of all  free
holomorphic automorphisms of the noncommutative ball $[B(\cH)^n]_1$.
We showed that any $\Psi\in \text{\rm Aut}(B(\cH)^n_1)$ has the form
$$
\Psi=\Phi_U \circ \Psi_\lambda,
$$
where $\Phi_U$ is an automorphism implemented by a unitary operator
$U$ on $\CC^n$, i.e.,
\begin{equation*}
 \Phi_U(X_1,\ldots,
X_n):=[X_1,\ldots, X_n]U , \qquad (X_1,\ldots, X_n)\in [B(\cH)^n]_1,
\end{equation*}
and $\Psi_\lambda$ is an involutive free holomorphic automorphism
associated with $\lambda:=\Psi^{-1} (0)\in \BB_n$. The  automorphism
$\Psi_\lambda:[B(\cH)^n]_1\to [B(\cH)^n]_1$  is   given by
\begin{equation*}
 \Psi_\lambda(X_1,\ldots, X_n):={
\lambda}-\Delta_{ \lambda}\left(I_\cH-\sum_{i=1}^n \bar{{
\lambda}}_i X_i\right)^{-1} [X_1,\ldots, X_n]
\Delta_{{\lambda}^*},\qquad (X_1,\ldots, X_n)\in [B(\cH)^n]_1,
\end{equation*}
where $\Delta_\lambda$ and $\Delta_{\lambda^*}$ are the defect
operators associated with  the row contraction
$\lambda:=[\lambda_1,\ldots, \lambda_n]$.
   Note that, when
$\lambda=0$, we have $\Psi_0(X)=-X$. We recall that if $\lambda \in
\BB_n\backslash \{0\}$ and   $\gamma:=\frac{1}{\|\lambda\|_2}$, then
$\Psi_\lambda$ is a free holomorphic function on $[B(\cH)^n]_\gamma$
which has the following properties:
\begin{enumerate}
\item[(i)]
$\Psi_\lambda (0)=\lambda$ and $\Psi_\lambda(\lambda)=0$;

\item[(ii)] $\Psi_\lambda$ is an involution, i.e., $\Psi_\lambda(\Psi_\lambda(X))=X$
for any $X\in [B(\cH)^n]_\gamma$;
\item[(iii)] $\Psi_\lambda$ is a free holomorphic automorphism of the
noncommutative unit ball $[B(\cH)^n]_1$;
\item[(iv)] $\Psi_\lambda$ is a homeomorphism of $[B(\cH)^n]_1^-$
onto $[B(\cH)^n]_1^-$.
\end{enumerate}

We say that a row contraction $T=(T_1,\ldots, T_n)\in
[B(\cH)^n]_1^-$ is homogeneous if  $T$ is unitarily equivalent to
$\varphi(T)$ for any $\varphi=(\varphi_1,\ldots, \varphi_n)\in
\text{\rm Aut}(B(\cH)^n_1)$.

\begin{theorem} \label{homo} Let $T:=[T_1,\ldots, T_n]\in [B(\cH)^n]_1^-$ be a
 completely non-coisometric row contraction. Then $T$ is homogeneous if and only if $\Theta_T\circ \Psi^{-1}$
 coincides with the characteristic function $\Theta_T$ for any $\Psi\in \text{\rm Aut}(B(\cH)^n_1)$.
\end{theorem}

\begin{proof}    Let $\Psi :=\Phi_U \circ
\Psi_\lambda$ be a  free holomorphic automorphism of $[B(\cH)^n]_1$,
where $U$ is a unitary operator on $\CC^n$ and $\lambda\in \BB_n$.
According to \cite{Po-automorphism},
    the   characteristic function has
the property that
$$
\Theta_{\Psi(T)}(X)=-(I_\cG\otimes \Omega^*)(\Theta_T\circ
\Psi^{-1})(X) (I_\cG \otimes \Omega_* {\bf U}), \qquad X\in
[B(\cG)^n]_1,
$$
 where $\Omega$ and $\Omega_*$ are the
unitary operators. Therefore,
 $\Theta_T\circ \Psi^{-1}$
coincides with the characteristic function $\Theta_{\Psi(T)}$ for
any
 $\Psi\in \text{\rm Aut}(B(\cH)^n_1)$.    Since $T$ and $\Psi(T)$
  are c.n.c. row contractions, we can apply Theorem 5.4 from \cite{Po-charact},
  to deduce that  $T$ is homogeneous if and only if $\Theta_T$ coincides with
  $\Theta_{\Psi(T)}$. Consequently,
 $T$ is homogeneous if and only if $\Theta_T$ coincides with $\Theta_T\circ \Psi^{-1}$
 for any $\Psi\in \text{\rm Aut}(B(\cH)^n_1)$.
The proof is complete.
\end{proof}

\begin{lemma} \label{prod} Let $\Phi_k, \Phi, \Gamma_p$, and  $\Gamma$ be
in the automorphism group $\text{\rm Aut}(B(\cH)^n_1)$, where $k,p\in \NN$.  If $\Phi_k\to
\Phi$ and $\Gamma_p\to \Gamma$ uniformly on $[B(\cH)^n]_1^-$, then
$\Phi_k\circ \Gamma_p\to \Phi\circ \Gamma$ uniformly on
$[B(\cH)^n]_1^-$, as $k,p\to\infty$.
\end{lemma}
\begin{proof}
Since $\Phi\in \text{\rm Aut}(B(\cH)^n_1)$, it is uniformly continuous on
$[B(\cH)^n]_1^-$. Hence, for any $\epsilon>0$, there is $\delta>0$ such that
$\|\Phi(Y)-\Phi(Z)\|<\frac{\epsilon}{2}$ for any $Y,Z\in [B(\cH)^n]_1^-$ with $\|Y-Z\|<\delta$. Taking into account that $\Gamma_p\to \Gamma$ uniformly on $[B(\cH)^n]_1^-$, we find $N\in \NN$ such that $\|\Gamma_p-\Gamma\|_\infty<\delta$ for any $p\geq N$.
Hence, we have
$$ \|\Phi(\Gamma_p(X))-\Phi(\Gamma(X))\|<\frac{\epsilon}{2}$$
for any $X\in [B(\cH)^n]_1^-$ and $p\geq N$.
Consequently, we have
\begin{equation*}
\begin{split}
\|(\Phi_k\circ \Gamma_p)(X)-(\Phi\circ \Gamma)(X)\|&\leq
\|(\Phi_k-\Phi)(\Gamma_p(X))\|+\|\Phi(\Gamma_p(X))-\Phi(\Gamma(X))\| \\
&\leq \|\Phi_k-\Phi\|_\infty+ \frac{\epsilon}{2}
\end{split}
\end{equation*}
for any $X\in [B(\cH)^n]_1^-$, $k\in\NN$, and $p\geq N$. Since $\|\Phi_k-\Phi\|_\infty\to 0$ as $k\to \infty$, there is $M\in \NN$ such that $\|\Phi_k-\Phi\|_\infty<\frac{\epsilon}{2}$ for any $k\geq M$. Combining these inequalities, we deduce that
$\|\Phi_k\circ \Gamma_p-\Phi\circ \Gamma\|_\infty <\epsilon$ for any $p\geq N$ and
$k\geq M$, which completes the proof.
\end{proof}

Let $\phi,\psi\in  \text{\rm Aut}(B(\cH)^n_1)$ and define
$$
d_\cE(\phi,\psi):=\|\phi -\psi\|_\infty +
\|\phi^{-1}(0)-\psi^{-1}(0)\|.
$$
One can easily check that $d_\cE$ is a metric on $\text{\rm Aut}(B(\cH)^n_1)$.

\begin{lemma} \label{conv-equi} Let $\Phi_k=\Phi_{U^{(k)}}\circ \Psi_{\lambda^{(k)}}$, $k\in \NN$,
and $\Phi=\Phi_U\circ \Psi_\lambda$ be free holomorphic automorphisms of the
 noncommutative ball $[B(\cH)^n]_1$, where $U^{(k)}, U\in \cU(\CC^n)$ and $\lambda^{(k)},\lambda\in \BB_n$.
Then the following statements are equivalent:
\begin{enumerate}
\item[(i)] $\Phi_k\to \Phi$ in the metric $d_\cE$;
\item[(ii)] $U^{(k)}\to U$ in $B(\CC^n)$ and $\lambda^{(k)}\to \lambda$
in the Euclidean norm of $\BB_n$;
    \item[(iii)] $\Phi_{U^{(k)}}\to \Phi_{U}$ and $\Psi_{\lambda^{(k)}}\to \Psi_\lambda$
     uniformly on  $[B(\cH)^n]_1^-$.
\end{enumerate}
\end{lemma}
\begin{proof}
First, we prove that (ii) is equivalent to (iii).
Assume that $U^{(k)}=[u_{ij}^{(k)}]_{n\times n}$, $k\in \NN$, and $U=[u_{ij}]_{n\times n}$ are unitary matrices with scalar entries, and  $\Phi_{U^{(k)}}\to \Phi_{U}$ uniformly on  $[B(\cH)^n]_1^-$, as $k\to\infty$.
For each $j=1,\ldots,n$, denote ${\bf I}_j:=[0,\ldots,I,\ldots, 0]$, where the identity is on the $j$-position.  Since
$\|\Phi_{U^{(k)}}({\bf I}_i)-\Phi_U({\bf I}_i)\|=\left(\sum_{j=1}^n |u_{ij}^{(k)}- u_{ij}|^2\right)^{1/2}$, it is clear that, for each $i,j\in \{1,\ldots,n\}$, $u_{ij}^{(k)}\to u_{ij}$ as $k\to\infty$. Hence, $U^{(k)}\to U$ in $B(\CC^n)$.
Conversely, assume that the latter condition holds. Since
$\|\Phi_{U^{(k)}}(X)-\Phi_U(X)\|\leq \|X\|\|U^{(k)}-U\|$ for any $X=[X_1,\ldots,X_n]\in [B(\cH)^n]_1^-$, we deduce that  $\Phi_{U^{(k)}}\to \Phi_{U}$ uniformly on
$[B(\cH)^n]_1^-$.

Now we prove that $\lambda^{(k)}\to \lambda$
in the Euclidean norm of $\BB_n$ if and only if $\Psi_{\lambda^{(k)}}\to \Psi_\lambda$
     uniformly on  $[B(\cH)^n]_1^-$. Since $\Psi_{\lambda^{(k)}}(0)=\lambda^{(k)}$ and $\Psi_\lambda(0)=\lambda$,  one implication is clear. To prove the converse, assume that $\lambda^{(k)}\to \lambda$
in the Euclidean norm of $\BB_n$.
Since the right creation operators $R_1,\ldots,R_n$ are isometries with orthogonal ranges, we have
$$
\left\|\sum_{i=1}^n \overline{\lambda}_i R_i\right\|=
\left\|\left( \sum_{i=1}^n {\lambda}_i R_i^*\right) \left(\sum_{i=1}^n \overline{\lambda}_i R_i\right)\right\|^{1/2}=\left(\sum_{i=1}^n |\lambda_i|^2\right)^{1/2}<1.
$$
Consequently, $\left(\sum_{i=1}^n \overline{\lambda}^{(k)}_i R_i\right)^{-1}$ converges to $\left(\sum_{i=1}^n \overline{\lambda}_i R_i\right)^{-1}$, as $k\to \infty$, in the operator norm.
Taking into account that
$$\widehat{\Psi}_\lambda=\lambda -\Delta_\lambda \left(I-\sum_{i=1}^n \overline{\lambda}_i R_i\right)^{-1}[R_1,\ldots,R_n]\Delta_{\lambda^*}
$$
and a similar relation holds for $\widehat{\Psi}_{\lambda^{(k)}}$, we deduce that
$\widehat{\Psi}_{\lambda^{(k)}}\to \widehat{\Psi}_\lambda$ in the operator norm.
Due to the noncommutative von Neumann inequality \cite{Po-von}, we have
$\|\Psi_{\lambda^{(k)}}(X)-\Psi_\lambda(X)\|\leq \|\widehat{\Psi}_{\lambda^{(k)}} - \widehat{\Psi}_\lambda\|
$
for any   $X=[X_1,\ldots,X_n]\in [B(\cH)^n]_1^-$. Hence,  $\Psi_{\lambda^{(k)}}\to \Psi_\lambda$
     uniformly on  $[B(\cH)^n]_1^-$, which proves our assertion. Therefore, (ii) is equivalent to (iii).

     Now, we prove that (i)$\implies$(ii). Assume that $d_\cE(\Phi_k,\Phi)\to 0$ as $k\to\infty$. Hence, $\Phi_k\to \Phi$ uniformly on $[B(\cH)^n]_1^-$ and
     $\lambda^{(k)}=\Phi_k^{-1}(0)\to \lambda=\Phi^{-1}(0)$ in $\BB_n$. Consequently, as proved above, we have that $\Psi_{\lambda^{(k)}}\to \Psi_\lambda$
     uniformly on  $[B(\cH)^n]_1^-$. Using Lemma \ref{prod} and the fact that
     $\Phi_k=\Phi_{U^{(k)}}\circ \Psi_{\lambda^{(k)}}$, $k\in \NN$,
and $\Phi=\Phi_U\circ \Psi_\lambda$, we deduce that
$$
\Phi_{U^{(k)}}=\Phi_k\circ \Psi_{\lambda^{(k)}}\to \Phi\circ \Psi_\lambda=\Phi_U
$$
uniformly on $[B(\cH)^n]_1^-$. Hence, $U^{(k)}\to U$ in $B(\CC^n)$ and, therefore, (ii) holds.

 It remains to prove that (ii)$\implies$(i).  Assume that (ii) holds. As proved above,
 $\Phi_{U^{(k)}}\to \Phi_{U}$ and $\Psi_{\lambda^{(k)}}\to \Psi_\lambda$
     uniformly on  $[B(\cH)^n]_1^-$. By Lemma \ref{prod}, we deduce that
     $$ \Phi_k=\Phi_{U^{(k)}} \circ \Psi_{\lambda^{(k)}}\to \Phi=\Phi_U\circ \Psi_\lambda
     $$
uniformly on  $[B(\cH)^n]_1^-$.  On the other hand, we have
$\Phi_k^{-1}(0)=\lambda^{(k)}\to \lambda=\Phi^{-1}(0)$ in $\BB_n$. Now, one can easily see that $d_\cE(\Phi_k,\Phi)\to 0$ as $k\to\infty$. The proof is complete.
\end{proof}

After these preliminaries, we can prove the following

\begin{theorem} \label{Aut}
The free holomorphic automorphism group \
 $\text{\rm Aut}(B(\cH)^n_1)$ is a   $\sigma$-compact, locally
compact topological group with respect to the topology induced by the metric $d_\cE$.
\end{theorem}
\begin{proof} First, we prove that the map
$$
\text{\rm Aut}(B(\cH)^n_1\times \text{\rm Aut}(B(\cH)^n_1\ni
(\Phi,\Gamma)\mapsto \Phi\circ\Gamma\in \text{\rm Aut}(B(\cH)^n_1 $$
is continuous when $\text{\rm Aut}(B(\cH)^n_1$ has the topology
induced by the metric $d_\cE$. For $k,p\in \NN$,  let

\begin{equation*}
\begin{split}
\Phi_k&=\Phi_{U^{(k)}}\circ \Psi_{\lambda^{(k)}},\quad
\Gamma_p=\Phi_{W^{(p)}}\circ \Psi_{\mu^{(p)}},\\
\Phi&=\Phi_U\circ \Psi_\lambda,\qquad \quad \ \,  \Gamma=\Phi_W\circ
\Psi_\mu,
\end{split}
\end{equation*}
be free holomorphic automorphisms of  $[B(\cH)^n]_1$, in
standard decomposition. Then $U^{(k)}, W^{(p)}, U,W$ are unitary
operators on $\CC^n$ and $\lambda^{(k)}, \mu^{(p)}, \lambda,\mu$ are
in $\BB_n$ satisfying relations
$$
\lambda^{(k)}=\Phi^{-1}_k(0),\ \mu^{(p)}=\Gamma_p^{-1}(0), \
\lambda=\Phi^{-1}(0),  \text{ and } \mu=\Gamma^{-1}(0).
$$
Since $\Phi_k\circ \Gamma_p\in \text{\rm Aut}(B(\cH)^n_1$, it has
the standard representation
\begin{equation}
 \label{phik}
 \Phi_k\circ
\Gamma_p=\Phi_{\Omega^{(kp)}}\circ \Psi_{z^{(kp)}}
\end{equation}
for some unitary operator $\Omega^{(kp)}\in \cU(\CC^n)$ and
$z^{(kp)}\in \BB_n$. Note that
$$
z^{(kp)}=(\Phi_k\circ \Gamma_p)^{-1}(0)= (\Psi_{\mu^{(p)}}^{-1}\circ
\Phi_{W^{(p)}}^{-1}\circ
\Phi_k^{-1})(0)=\Psi_{\mu^{(p)}}\left(\lambda^{(k)}
{W^{(p)}}^*\right).
$$
Similarly, since $\Phi\circ \Gamma \in \text{\rm Aut}(B(\cH)^n_1$,
we have $\Phi\circ \Gamma=\Phi_\Omega \circ \Psi_z$ for some $\Omega\in
\cU(\CC^n)$ and $z=\Psi_\mu(\lambda W^*)\in \BB_n$. Assume that
$d_\cE(\Phi_k, \Phi)\to 0$ as $k\to\infty$ and
$d_\cE(\Gamma_p,\Gamma)\to 0$ as $p\to \infty$.  According to Lemma
\ref{conv-equi}, $\lambda^{(k)}\to \lambda$ in $\BB_n$ and
$W^{(p)}\to W$ in $B(\CC^n)$.  Hence, $\lambda^{(k)} {W^{(p)}}^*\to
\lambda W^*$ in $B(\CC^n)$. Applying again Lemma \ref{conv-equi}, we
deduce that $\Psi_{\mu^{(p)}}\to \Psi_\mu$ uniformly  on
$[B(\cH)^n]_1^-$. Consequently,

$$
z^{(kp)}=\Psi_{\mu^{(p)}}\left(\lambda^{(k)} {W^{(p)}}^*\right)\to
z=\Psi_\mu(\lambda W^*)\in \BB_n
$$
as  $k,p\to \infty$. This implies that $\Psi_{z^{(kp)}}\to \Psi_z$
uniformly on $[B(\cH)^n]_1^-$. On the other hand, since $\Phi_k\to
\Phi$ and $\Gamma_p\to \Gamma$ uniformly on $[B(\cH)^n]_1^-$, Lemma
\ref{prod} shows that $\Phi_k\circ \Gamma_p\to \Phi\circ \Gamma$
uniformly on $[B(\cH)^n]_1^-$ as $k,p\to \infty$. Now, by  relation \eqref{phik} and Lemma
\ref{prod}, we deduce that
$$
\Phi_{\Omega^{(kp)}}=(\Phi_k\circ \Gamma_p)\circ \Psi_{z^{(kp)}} \to
(\Phi\circ \Gamma)\circ \Psi_z=\Phi_\Omega
$$
uniformly on $[B(\cH)^n]_1^-$.  This implies that $\Omega^{(kp)}\to
\Omega$ in $B(\CC^n)$ as $k,p\to\infty$. Using again Lemma
\ref{conv-equi}, we conclude that $\Phi_k\circ \Gamma_p\to \Phi\circ
\Gamma$, which proves our assertion.

  In what follows, we show that the map $\Phi\mapsto \Phi^{-1}$ is
  continuous on $\text{\rm Aut}(B(\cH)^n_1)$ with the topology
  induced by the metric $d_\cE$.
Assume that   $d_\cE(\Phi_k, \Phi)\to 0$  as $k\to\infty$.
Using the same notations as above, we have
$\Phi_{U^{(k)}}\to \Phi_U$ and $\Psi_{\lambda^{(k)}}\to \Psi_\lambda$ uniformly on $[B(\cH)^n]^-_1$. Applying Lemma \ref{prod}, we deduce that
\begin{equation}
\label{phi-inv}
\Phi_k^{-1}=\Psi_{\lambda^{(k)}}\circ \Phi_{{U^{(k)}}^*}
\to \Psi_\lambda\circ \Phi_{U^*}=\Phi^{-1}
\end{equation}
uniformly on $[B(\cH)^n]^-_1$, as $k\to\infty$.
On the other hand, we have the standard representations
$\Phi_k^{-1}=\Phi_{W^{(k)}}\circ \Psi_{z^{(k)}}$ and
$\Phi_k^{-1}=\Phi_{W}\circ \Psi_{z}$ for some unitary operators $W^{(k)},W\in B(\CC^n)$ and $z^{(k)}, z\in \BB_n$. Note that
$z^{(k)}=\Phi_k(0)=(\Phi_{U^{(k)}}\circ \Psi_{\lambda^{(k)}})(0)=\lambda^{(k)} U^{(k)}$ and $z=\Phi(0)=\lambda U$. Since $\lambda^{(k)}\to \lambda$ in $\BB_n$, we have $z^{(k)}\to z$ in $\BB_n$, which implies $\Psi_{z^{(k)}}\to \Psi_z$ uniformly
on $[B(\cH)^n]^-_1$, as $k\to\infty$.
Using relation \eqref{phi-inv} and Lemma \ref{prod}, we deduce that
$$
\Phi_{W^{(k)}}=\Phi_k^{-1}\circ \Phi_{z^{(k)}}\to \Phi_W=\Phi^{-1}\circ \Psi_z
$$
uniformly on $[B(\cH)^n]^-_1$. Applying  Lemma \ref{conv-equi}, we conclude that $\Phi_k^{-1}\to \Phi^{-1}$  in the  topology
induced by the metric $d_\cE$.

Each free holomorphic automorphism $\Phi \in \text{\rm Aut}(B(\cH)^n_1$ has a unique representation
$\Phi=\Phi_U\circ \Psi_\lambda$, where $\lambda:=\Phi^{-1}(0)$ and  $U\in \cU(\CC^n)$.
This generates  a bijection
$\chi: \text{\rm Aut}(B(\cH)^n_1\to \cU(\CC^n)\times \BB_n$ by setting
$\chi(\Phi):=(U,\lambda)$.
According to Lemma \ref{conv-equi}, the map $\chi$ is a homeomorphism of topological spaces, where $\text{\rm Aut}(B(\cH)^n_1$ has the topology induced by the metric $d_\cE$ and  $\cU(\CC^n)\times \BB_n$ has the natural topology. Consequently, since
$\cU(\CC^n)\times \BB_n$ is a $\sigma$-compact, locally compact topological space, so is  the automorphism  group $\text{\rm Aut}(B(\cH)^n_1$. The proof is complete.
\end{proof}

\begin{corollary} The free holomorphic automorphism group \
 $\text{\rm Aut}(B(\cH)^n_1)$ is  path connected.
\end{corollary}
\begin{proof}
Fix a unitary operator $U\in \cU(\CC^n)$ and $\lambda\in \BB_n$. Since the unitary group $\cU(\CC^n)$ is path connected, there is a a continuous map
$[0,1]\ni t\mapsto U_t\in \cU(\CC^n)$ such that $U_0=I$ and $U_1=U$. Using Lemma \ref{conv-equi}, we deduce that the map $\varphi:[0,1]\to \text{\rm Aut}(B(\cH)^n_1)$ defined by
$\varphi(t):= \Phi_{U_t}\circ \Psi_{t\lambda}$
is continuous with respect to the metric $d_\cE$. Since $\varphi(0)=\Psi_0$ and $\varphi(1)=\Phi_U\circ \Psi_\lambda$, the proof is complete.
\end{proof}

Let $\text{\rm Aut}(B(\cH)^n_1)$ be the free holomorphic
automorphism group of the noncommutative ball $[B(\cH)^n]_1$ and let
$\cU(\cK)$ be the unitary group on the Hilbert space $\cK$.
According to Theorem \ref{Aut}, $\text{\rm Aut}(B(\cH)^n_1)$ is a topological group with respect to the metric $d_\cE$.
A  map
$\pi: \text{\rm Aut}(B(\cH)^n_1)\to \cU(\cK)$  is called (unitary)
projective representation if the following conditions are satisfied:
\begin{enumerate}
\item[(i)] $\pi(id)=I$, where $id$ is the identity on $[B(\cH)^n]_1$;

\item[(ii)] $
 \pi(\varphi) \pi(\psi)=c(\varphi,\psi) \pi({\varphi\circ \psi})$,
  for any  $\varphi,\psi\in \text{\rm Aut}(B(\cH)^n_1)$, where $c(\varphi,\psi)$
  is a complex number  with
 $|c(\varphi,\psi)|=1$;
\item[(iii)]
 the map $\text{\rm Aut}(B(\cH)^n_1)\ni \varphi\mapsto \left<\pi(\varphi)\xi,\eta\right> \in \CC$
is continuous for each $\xi,\eta\in \cK$.
\end{enumerate}

\begin{theorem}\label{projective} Any completely non-coisometric row contraction $T:=[T_1,\ldots, T_n]\in [B(\cH)^n]_1^-$ with
 constant characteristic function is homogeneous. If $T$
 is irreducible, then the following statements hold:
 \begin{enumerate}
\item[(i)] $
\varphi_i(T)=U_\varphi^* T_i U_\varphi$ for all
$\varphi\in \text{\rm Aut}(B(\cH)^n_1),
$
where  $U_\varphi\in B(F^2(H_n))$ is a unitary operator and
$$U_\varphi U_\psi=c(\varphi,\psi) U_{\varphi\circ \psi},\qquad \varphi,\psi\in \text{\rm Aut}(B(\cH)^n_1),
$$
 for some complex number  $c(\varphi,\psi)\in \TT$.
\item[(ii)]the map $\varphi\to U_\varphi^*$ is continuous from the uniform topology to the strong operator topology.
    \item[(iii)] The map $\pi: \text{\rm Aut}(B(\cH)^n_1)\to B(F^2(H_n))$ defined by $\pi(\varphi):=U_\varphi$ is
a
projective representation  of the automorphism group $\text{\rm Aut}(B(\cH)^n_1)$.
\end{enumerate}
\end{theorem}
\begin{proof} The fact that $T$ is homogeneous follows from Theorem \ref{homo} using the fact that the characteristic function is constant, i.e., $\Theta_T=\Theta_T(0)$.
According to Theorem \ref{charact-const},
if $T:=[T_1,\ldots, T_n]\in [B(\cH)^n]_1^-$ is a
 c.n.c row contraction with  polynomial characteristic function,
   then  the characteristic function $\Theta_T$
is a constant  if and only if $T$ is a pure isometry.
Consequently,  if $T$ is irreducible   we can assume, without loss of generality, that $T=[S_1,\ldots, S_n]$.

Let $\varphi=(\varphi_1,\ldots, \varphi_n)\in \text{\rm Aut}(B(\cH)^n_1)$ and let $\widehat \varphi=(\widehat
\varphi_1,\ldots, \widehat\varphi_n)$ be its model boundary
function.  Note that $\widehat \varphi$ is a pure row isometry and
and $\widehat \varphi_i=\varphi_i(S_1,\ldots,S_n)$ for $i=1,\ldots,
n$. Using the noncommutative Poisson transform at $\widehat
\varphi$, we obtain
\begin{equation}\label{Poiss}
\varphi_i(S_1,\ldots, S_n)=P_{\widehat \varphi}(S_i)=K_{\widehat
\varphi}^*(S_i\otimes I_{\cD_{\widehat \varphi})})K_{\widehat
\varphi}, \qquad i=1,\ldots,n,
\end{equation}
where the Poisson kernel $K_{\widehat \varphi}:F^2(H_n)\to
  F^2(H_n)\otimes \cD_{\widehat \varphi}$ is an isometry. On the
other hand, since $\widehat \varphi^* \widehat \varphi=I$, the
characteristic function $\tilde \Theta_{\widehat \varphi}=0$. Since
$I-\tilde{\Theta}_{\widehat \varphi} \tilde{\Theta}_{\widehat
\varphi}^*=K_{\widehat \varphi}K_{\widehat \varphi}^*$, we have
$K_{\widehat \varphi}K_{\widehat \varphi}^*=I$, which implies that
$K_{\widehat \varphi}$ is a unitary operator.

According to   \cite{Po-automorphism} , $ \varphi=\Phi_U\circ
\Psi_\lambda$, where
$\lambda:=(\lambda_1,\ldots, \lambda)=\varphi^{-1}(0)\in \BB_n$ and  $U$ is  unitary operator  on $\CC^n$. Moreover, we have
$$\Delta_{\hat\Psi_\lambda}^2=\Delta_\lambda \left(I-\sum_{i=1}^n \bar
\lambda S_i\right)^{-1} P_\CC \left(I-\sum_{i=1}^n \lambda_i
S_i^*\right)^{-1} \Delta_\lambda.
$$
Therefore, there is a unitary operator $\Lambda_\lambda:\cD_{{\hat
\Psi_\lambda}}\to \CC$ defined by
\begin{equation*}
\begin{split}
\Lambda_\lambda\Delta_{\hat \Psi_\lambda}
f&:=(1-\|\lambda\|_2^2)^{1/2} P_\CC
\left(I-\sum_{i=1}^n \lambda_i S_i^*\right)^{-1} f\\
&=(1-\|\lambda\|_2^2)^{1/2} f(\lambda)
\end{split}
\end{equation*}
for any $ f\in F^2(H_n)$.  Hence, we deduce that
 $$\Lambda_\lambda
(z\Delta_{\hat \Psi_\lambda} (1))=z(1-\|\lambda\|_2^2)^{1/2}, \qquad
z\in \CC, $$
  and $\cD_{{\hat \Psi_\lambda}}=\CC \Delta_{\hat
\Psi_\lambda} (1)$. Since $\Delta_{\widehat \varphi}=\Delta_{\hat
\Psi_\lambda}$,  we deduce that the operator $W_{\widehat \varphi}:
F^2(\cH_n)\otimes \cD_{\widehat \varphi}\to F^2(H_n)$ defined by
$$
W_{\widehat \varphi}(g\otimes z\Delta_{\widehat \varphi}(1)):=
z(1-\|\lambda\|_2^2)^{1/2}  g,\qquad  g\in F^2(H_n) \text{ and  }
z\in \CC,
$$
is unitary. Consequently, we have
\begin{equation}
\label{W*}
W_{\widehat \varphi}^*(g)=g\otimes
\frac{1}{(1-\|\lambda\|_2^2)^{1/2}}\Delta_{\widehat
\varphi}(1)),\qquad g\in F^2(H_n).
\end{equation}
Setting  $U_\varphi:= W_{\widehat\varphi} K_{\widehat \varphi}$,
relation \eqref{Poiss} implies
\begin{equation*}
\varphi_i(S_1,\ldots, S_n) =U_\varphi^* S_i U_\varphi,\qquad
i=1,\ldots,n,
\end{equation*}
for any  $\varphi\in \text{\rm Aut}(B(\cH)^n_1)$. Hence, if $\psi\in
\text{\rm Aut}(B(\cH)^n_1)$, then
\begin{equation}\label{Poiss2}
(\varphi_i\circ\psi)(S_1,\ldots, S_n) =U_{\varphi\circ \psi}^* S_i
U_{\varphi\circ \psi},\qquad i=1,\ldots,n.
\end{equation}
On the other hand, due to Theorem 3.1 from \cite{Po-automorphism},
the noncommutative Poisson transform satisfies the relation
$P_{\widehat{\varphi\circ \psi}}[\chi]=P_{\widehat \psi} P_{\widehat
\varphi}[\chi]$ for any $\chi\in C^*(S_1,\ldots,S_n)$, the Cuntz-Toeplitz $C^*$-algebra generated by the left creation operators $S_1,\ldots, S_n$. In
particular, when $\chi=S_i$,  we obtain
$$K_{\widehat{\varphi\circ
\psi}}^*(S_i\otimes I_{\cD_{{\widehat{\varphi\circ
\psi}}}})K_{\widehat{\varphi\circ \psi}} =K_{\widehat{ \psi}}^*\left\{[
K_{\widehat{\varphi }}^*(S_i\otimes I_{\cD_{\widehat{\varphi
}}})K_{\widehat{\varphi }}]\otimes I_{\cD_{\widehat{\psi }}}\right\}
K_{\widehat{ \psi}},\qquad i=1,\ldots,n.
$$
Hence, we deduce that
$$
 (\varphi_i\circ\psi)(S_1,\ldots, S_n)= U_{\psi}^* U_\varphi^* S_i U_\varphi U_{\psi},\qquad i=1,\ldots,n.
 $$
Combining this relation with \eqref{Poiss2}, we deduce that
$$
U_{\varphi\circ \psi}^* S_i U_{\varphi\circ \psi}=U_{\psi}^*
U_\varphi^* S_i U_\varphi U_{\psi},\qquad i=1,\ldots,n,
$$
which is equivalent to
$$
U_\varphi U_{\psi}U_{\varphi\circ \psi}^* S_i =  S_i U_\varphi
U_{\psi}U_{\varphi\circ \psi}^*,\qquad i=1,\ldots,n.
$$
Since $S_1,\ldots, S_n$ is irreducible and $U_\varphi
U_{\psi}U_{\varphi\circ \psi}^*$ is a unitary operator, we have
$U_\varphi U_{\psi}U_{\varphi\circ
\psi}^*=\overline{c(\varphi,\psi)} I$ for some complex number with
$|c(\varphi,\psi)|=1$. Hence, we deduce that $
 U_\varphi U_\psi=c(\varphi,\psi) U_{\varphi\circ \psi}
$ for any $\varphi,\psi\in \text{\rm Aut}(B(\cH)^n_1)$.

Now, we prove part (ii).
Let $\varphi^{(p)}:=(\varphi_1^{(p)},\ldots, \varphi_n^{(p)})$, $p=1,2,\ldots$, and
$\psi:=(\psi_1,\ldots, \psi_n)$  be in  $\text{\rm Aut}(B(\cH)^n_1)$ such that $\varphi_i^{(p)}$ converges to $\psi_i$ in the uniform norm on $[B(\cH)^n]_1$, that is,
$$\|\varphi_i^{(p)}-\psi_i\|_\infty:=\sup_{X\in [B(\cH)^n]_1}\|\varphi_i^{(p)}(X)-\psi_i(X)\|\to 0\qquad \text{ as } p\to\infty,
$$
for $i=1,\ldots,n$. Since $\varphi_i^{(p)}$ and  $\psi_i$ are
uniformly continuous  on $[B(\cH)^n]_1$,  the model boundary
functions $\widehat{\varphi_i^{(p)}}$ and $\widehat\psi_i$ are in
the noncommutative disc algebra $\cA_n$ and  we have
$\widehat{\varphi_i^{(p)}}=\varphi_i^{(p)}(S_1,\ldots,S_n)$ and
$\widehat\psi_i=\psi(S_1,\ldots,S_n)$. Consequently, the convergence
above implies that $ \widehat{\varphi_i^{(p)}}\to \widehat\psi_i$ in
the operator norm topology. Each $\varphi^{(p)}\in \text{\rm Aut}(B(\cH)^n_1)$
has the form $ \varphi^{(p)}=\Phi_{U_p} \circ \Psi_{\lambda^{(p)}},
$ where $\Phi_{U_p}$ is an automorphism implemented by a unitary
operator $U_p$ on $\CC^n$, i.e.,
\begin{equation*}
 \Phi_{U_p}(X_1,\ldots,
X_n):=[X_1,\ldots, X_n]U_p , \qquad (X_1,\ldots, X_n)\in [B(\cH)^n]_1,
\end{equation*}
and $\Psi_{\lambda^{(p)}}$ is the involutive free holomorphic   automorphism
associated with $\lambda^{(p)}:=(\varphi^{(p)})^{-1} (0)\in \BB_n$.
Similarly, we have $\psi=\Phi_U\circ \Psi_\mu$, where $U\in B(\CC^n)$   is a unitary operator and $\Psi_\mu$ is the involutive free holomorphic   automorphism
associated with $\mu:=\psi^{-1} (0)\in \BB_n$.
Due to the above-mentioned convergences,  we deduce that
$\varphi^{(p)}(0)\to \psi(0)$ as $p\to\infty$. Taking into account that
$\Psi_{\lambda^{(p)}}(0)=\lambda^{(p)}$ and $\Psi_\mu(0)=\mu$,
we have $\varphi^{(p)}(0)=\lambda^{(p)} U_p$ and $\psi(0)=\mu U$. Therefore,
$\lambda^{(p)} U_p$ converges to $\mu U$ in the operator norm. Since $U_p$ and $U$ are unitary operators, we deduce that
$\|\lambda^{(p)}\|_2\to \|\mu\|_2$ as $p\to\infty$.

Given $\epsilon>0$ and $x=\sum_{\alpha\in \FF_n^+} a_\alpha e_\alpha\in F^2(H_n)$,  let $k\in \NN$ be such that
$\|x-\sum_{\alpha\in \FF_n^+, |\alpha|\leq k} a_\alpha e_\alpha\|<\frac{\epsilon}{4}$.
Using relation  \eqref{W*} and the properties of the noncommutative Poisson kernel, we have
\begin{equation*}
\begin{split}
\sum_{|\alpha|\leq k}a_\alpha U^*_{\varphi^{(p)}} e_\alpha
&=\sum_{|\alpha|\leq k} a_\alpha K^*_{\widehat{\varphi^{(p)}}} W^*_{\widehat{\varphi^{(p)}}} e_\alpha\\
&=\sum_{|\alpha|\leq k} a_\alpha K^*_{\widehat{\varphi^{(p)}}}
\left(e_\alpha\otimes
\frac{1}{(1-\|\lambda^{(p)}\|_2^2)^{1/2}}\Delta_{\widehat
{\varphi^{(p)}}}(1)\right)\\
&=
\sum_{|\alpha|\leq k} a_\alpha
\frac{1}{(1-\|\lambda^{(p)}\|_2^2)^{1/2}} [\widehat{\varphi^{(p)}}]_\alpha\Delta^2_{\widehat
{\varphi^{(p)}}}(1).
\end{split}
\end{equation*}
A similar relation holds if  we replace  $\varphi^{(p)}$ with $\psi$.
Since  $
\widehat{\varphi_i^{(p)}}\to \widehat\psi_i$ in the operator norm topology and
$\|\lambda^{(p)}\|_2\to \|\mu\|_2$ as $p\to\infty$, there is $N\in \NN$ such that
$$
\left\|\sum_{|\alpha|\leq k}a_\alpha U^*_{\varphi^{(p)}} e_\alpha -\sum_{|\alpha|\leq k}a_\alpha U^*_{\psi} e_\alpha\right\|<\frac{\epsilon}{2}
$$
for all $p\geq N$. Using the fact that  $U_{\varphi^{(p)}}$ and $U_\psi$ are unitary operators, we deduce that
\begin{equation*}
\begin{split}
\|U^*_{\varphi^{(p)}} x-U^*_{\psi}x\|
&\leq \left\|U^*_{\varphi^{(p)}} \left(x-\sum_{ |\alpha|\leq k} a_\alpha e_\alpha\right)\right\|
+
\left\|U^*_{\varphi^{(p)}}\left(\sum_{|\alpha|\leq k} a_\alpha e_\alpha\right) -U_\psi^*\left(\sum_{ |\alpha|\leq k} a_\alpha e_\alpha\right)\right\|\\
&\qquad \qquad + \left\|U^*_\psi \left(x-\sum_{ |\alpha|\leq k} a_\alpha e_\alpha\right)\right\|\\
&\leq 2\left\|x-\sum_{ |\alpha|\leq k} a_\alpha e_\alpha\right\| + \left\|U^*_{\varphi^{(p)}}\left(\sum_{|\alpha|\leq k} a_\alpha e_\alpha\right) -U_\psi^*\left(\sum_{ |\alpha|\leq k} a_\alpha e_\alpha\right)\right\|\\
&\leq 2\frac{\epsilon}{4}+\frac{\epsilon}{2}=\epsilon
\end{split}
\end{equation*}
for any $p\geq N$. Therefore  $U^*_{\varphi^{(p)}}$ converges to $U_\psi^*$, as $p\to \infty$, in the strong operator topology.

To prove part (iii), let $\varphi^{(p)}, \varphi$ be in $\text{\rm
Aut}(B(\cH)^n_1)$ be such that $\varphi^{(p)}\to \varphi$ in the metric $d_\cE$, as $p\to \infty$. Then  $\|\varphi^{(p)}- \varphi\|_\infty\to 0$, as $p\to\infty$. Using  (i) and (ii), we deduce   that  the map $\pi: \text{\rm Aut}(B(\cH)^n_1)\to B(F^2(H_n))$ defined by $\pi(\varphi):=U_\varphi$ is
a
projective representation  of the automorphism group $\text{\rm Aut}(B(\cH)^n_1)$.
The proof is complete.
\end{proof}

We say that two projective representations $\pi_1$, $\pi_2$ of
$\text{\rm Aut}(B(\cH)^n_1)$ on the Hilbert spaces  $\cH_1$ and
$\cH_2$, respectively, are  equivalent if there exists a unitary
operator $U:\cH_1\to \cH_2$ and a Borel function $\sigma:\text{\rm
Aut}(B(\cH)^n_1)\to \TT$ such that $\pi_2(\varphi) U=\sigma(\varphi)
U \pi_1(\varphi)$ for all $\varphi\in \text{\rm Aut}(B(\cH)^n_1)$.

We remark  that  if $\pi_1$  and $\pi_2$ are two projective representations of
$\text{\rm Aut}(B(\cH)^n_1$  associated with $T$, as in Theorem \ref{projective}, then we have
$\varphi_i(T)=\pi_1(\varphi)^* T_i \pi_1(\varphi)$  and
$\varphi_i(T)=\pi_2(\varphi)^* T_i \pi_2(\varphi)$
for all
$\varphi\in \text{\rm Aut}(B(\cH)^n_1)
$
and $i=1,\ldots,n$. Hence,  $\pi_1(\varphi)\pi_2(\varphi)^*$ commutes with each operator $T_1,\ldots, T_n$.  Since $[T_1,\ldots, T_n]$ is irreducible, we deduce that
$\pi_1(\varphi)\pi_2(\varphi)^*=d(\varphi) I$ for some constant $d(\varphi)\in \TT$ which proves that $\pi_1$ and $\pi_2$ are  equivalent.

\bigskip

       %

      \end{document}